\documentclass{article}

\usepackage[left=3cm, right=3cm, top=3cm, bottom=3cm]{geometry}

\usepackage{amsmath,yhmath}
\usepackage{amssymb,amsthm}
\usepackage{mathrsfs}
\usepackage{hyperref}
\usepackage{amsfonts}
\usepackage{graphicx}
\usepackage{epstopdf}
\usepackage{overpic}
\usepackage{multirow}
\usepackage{xcolor}
\usepackage[misc]{ifsym}
\usepackage{enumitem}
\usepackage{pgfplots}
\usepackage[overload]{empheq}

\definecolor{frenchblue}{rgb}{0.0, 0.45, 0.73}

\usepackage{algpseudocode}
\algnewcommand{\LineComment}[1]{\State {\color{frenchblue}\(\triangleright\) #1}}
\usepackage{algorithm}

\def\R{\mathbb{R}}

\usepackage{tikz}
\usetikzlibrary{shapes,arrows}
\usetikzlibrary{calc,patterns,angles,quotes}
\usetikzlibrary{matrix}
\usepackage{tkz-euclide}
\usetikzlibrary{tikzmark}

\def\0{\mathbf{0}}
\def\1{\mathbf{1}}

\def\f{\mathbf{f}}

\def\h{\mathbf{h}}
\def\v{\mathbf{v}}

\def\R{\mathbb{R}}

\def\E{\mathcal{E}}
\def\F{\mathcal{F}}

\def\M{\mathcal{M}}

\def\V{\mathcal{V}}
\def\I{\mathtt{I}}
\def\B{\mathtt{B}}
\def\C{\mathtt{C}}
\def\X{\mathtt{X}}
\def\Y{\mathtt{Y}}

\def\trace{\mathrm{trace}}

\newcommand{\argmin}{\operatornamewithlimits{argmin}}

\newtheorem{theorem}{Theorem}
\newtheorem{corollary}{Corollary}

\title{Theoretical Foundation of the Stretch Energy Minimization for Area-Preserving Mappings}
\author{Mei-Heng Yueh\\
\href{mailto:yue@ntnu.edu.tw}{yue@ntnu.edu.tw} }
\date{Department of Mathematics\\ National Taiwan Normal University}

\begin{document}

\maketitle

\begin{abstract}
The stretch energy is a fully nonlinear energy functional that has been applied to the numerical computation of area-preserving mappings. However, this approach lacks theoretical support and the analysis is complicated due to the full nonlinearity of the functional. In this paper, we provide a theoretical foundation of the stretch energy minimization for the computation of area-preserving mappings, including a neat formulation of the gradient of the functional, and the proof of the minimizers of the functional being area-preserving mappings. In addition, the geometric interpretation of the stretch energy is also provided to better understand this energy functional. Furthermore, numerical experiments are demonstrated to validate the effectiveness and accuracy of the stretch energy minimization for the computation of square-shaped area-preserving mappings of simplicial surfaces. 
\end{abstract}

\section{Introduction}
\label{intro}

An area-preserving mapping is also called an equiareal mapping or authalic mapping, which is a bijective mapping that preserves the area. In application aspects, area-preserving mappings could serve as parameterizations of surfaces in 3D space. It has been applied to processing issues in computer vision and graphics such as resampling, remeshing, and registration of surface mesh models \cite{FlHo05,ShPR06,HoLe07}. Classical methods for computing area-preserving mapping includes stretch-minimizing method \cite{SaSG01,YoBS04}, Lie advection method \cite{ZoHG11}, optimal mass transportation method \cite{DoTa10,ZhSG13,SuCQ16}, and diffusion-based method \cite{ChRy18}.

In recent years, a series of efficient numerical algorithms \cite{Yueh17,YuLW17,YuLW19,YuLL19,YuHL20,YuLL20,KuLY21} for the computation of parameterizations of surfaces and $3$-manifolds have been well-developed based on minimizing nonlinear energy functionals of the form 
$$
E(f) = \frac{1}{2} \trace\left( 
\begin{bmatrix}
f(v_1) & \cdots & f(v_n)
\end{bmatrix}
L(f) \begin{bmatrix}
f(v_1)^\top \\
\vdots\\
f(v_n)^\top
\end{bmatrix} \right),
$$
where $f:\M\to\R^d$ is a simplicial mapping defined on a simplicial $d$-complex with the image $f(v_\ell)$ being a column vector in $\R^d$, for $\ell=1, \ldots, n$, and $L(f)$ is a Laplacian matrix with weights being dependent on $f$. In particular, the area-preserving mapping is computed by minimizing the stretch energy \cite{Yueh17,YuLW19}. Noting that the stretch energy functional is fully nonlinear, still, the minimizer can be effectively obtained by iteratively approximating the critical point of the functional, and the resulting mapping preserves the area well. Compared to other state-of-the-art algorithms, stretch energy minimization for area-preserving mappings has advantages in both effectiveness and accuracy. However, this approach still lacks a theoretical foundation.

\paragraph{Contribution}

In this paper, we establish the theoretical foundation of the stretch energy minimization for the computation of area-preserving mappings. 
The contributions of this paper are three-fold. 
\begin{itemize}
\item[(i)] We derive a neat gradient formula of the stretch energy functional that is easy to implement. 
\item[(ii)] We provide the geometric interpretation of the stretch energy functional. 
\item[(iii)] We prove that minimizers of the stretch energy functional are area-preserving mappings. 
\end{itemize}
In addition, we demonstrate the associated efficient energy minimization algorithm for the computation of square-shaped area-preserving mappings of simply connected open surfaces to validate the effectiveness and accuracy of the stretch energy minimization.

\paragraph{Notations}

In this paper, we use the following notations.

\begin{itemize}[label={$\bullet$}]
\item Bold letters, for instance, $\mathbf{f}$, denote real-valued vectors or matrices.
\item Capital letters, for instance, $L$, denote real-valued matrices.
\item Typewriter letters, for instance, $\mathtt{I}$, $\mathtt{B}$, denote ordered sets of indices. 
\item $\mathbf{f}_i$ denotes the $i$th entry or row of the vector or matrix $\mathbf{f}$.
\item $\mathbf{f}_\mathtt{I}$ denotes the subvector or submatrix of $\mathbf{f}$ composed of $\mathbf{f}_i$, for $i\in\mathtt{I}$.
\item $L_{i,j}$ denotes the $(i,j)$th entry of the matrix $L$.
\item $L_{\mathtt{I},\mathtt{J}}$ denotes the submatrix of $L$ composed of $L_{i,j}$, for $i\in\mathtt{I}$ and $j\in\mathtt{J}$.
\item $\mathbb{R}$ denotes the set of real numbers. 
\item $\left[{v}_0, \ldots, {v}_k\right]$ denotes the $k$-simplex with vertices ${v}_0, \ldots, {v}_k$. 
\item $|[{v}_0, \ldots, {v}_k]|$ denotes the volume of the $k$-simplex $\left[v_0, \ldots, v_k\right]$.
\item $\M$ denotes a simplicial $2$-complex of $n$ vertices and $m$ triangular faces. 
\item $\mathbf{0}$ and $\mathbf{1}$ denote the zero and one vectors or matrices of appropriate sizes.
\end{itemize}

\paragraph{Outline}
The remaining parts of the paper is organized as follows. 
In Section \ref{sec:SimplicialSurface}, we introduce the notation of simplicial surfaces and mappings as well as their matrix representations. Then, we provide the theoretical foundation of the stretch energy minimization in Section \ref{sec:SEM}. 
The associated algorithm for the computation of area-preserving mappings is demonstrated in Section \ref{sec:SEMAlg}. 
Numerical results are demonstrated and discussed in Subsection \ref{sec:Numerical} to validate the effectiveness and accuracy of the stretch energy minimization for the computation of square-shaped area-preserving mappings. 
Ultimately, concluding remarks are given in Section \ref{sec:CR}.

\section{Simplicial surfaces and mappings} \label{sec:SimplicialSurface}

In numerical computation, surfaces and mappings are usually approximated by simplicial surfaces and mappings. 
A simplicial surface can be represented as a triangular mesh composed of vertices, edges, and triangular faces, and a simplicial mapping can be interpreted as a change of embedding of vertices of the simplicial surface.

\subsection{Simplicial surfaces}

A simplicial surface or triangular mesh is a simplicial $2$-complex $\M$ composed of vertices 
\[
\V(\M) = \left\{\v_\ell = (v_\ell^1, v_\ell^2, v_\ell^3 )^\top \in \mathbb{R}^3 \right\}_{\ell=1}^n,
\]
edges
$$
\E(\M) = \left\{ [\v_i,\v_j] \subset\R^3 \right\},
$$
and triangular faces
\[
\F(\M) = \left\{ \tau_\ell = [\v_{i_\ell}, \v_{j_\ell}, \v_{k_\ell}] \subset \mathbb{R}^3\right\}_{\ell=1}^m,
\]
where the bracket $[\v_{i_\ell}, \v_{j_\ell}, \v_{k_\ell}]$ denotes the convex hull of vertices $\{ \v_{i_\ell}, \v_{j_\ell}, \v_{k_\ell} \}$, defined as 
\[
\tau_\ell = [\v_{i_\ell}, \v_{j_\ell}, \v_{k_\ell}]
:= \left\{\alpha_{1} \v_{i_\ell} + \alpha_{2} \v_{j_\ell} + \alpha_{3} \v_{k_\ell} \mid \alpha_1+\alpha_2+\alpha_3 = 1 \right\}\subset\mathbb{R}^3.
\] 
In other words, a simplicial surface $\M=\cup_{\ell=1}^m \tau_\ell$ is a surface composed of triangular faces. 
In practice, the sets of vertices $\V(\M)$ and triangular faces $\F(\M)$ are represented as matrices
\[
V = \begin{bmatrix}
\v_1^\top \\
\vdots \\
\v_n^\top
\end{bmatrix}
= \begin{bmatrix}
v_1^1 & v_1^2 & v_1^3\\
\vdots & \vdots & \vdots \\
v_n^1 & v_n^2 & v_n^3
\end{bmatrix}
:=\begin{bmatrix}
\v^1 & \v^2 & \v^3
\end{bmatrix}
~~
\text{and}
~~
F = \begin{bmatrix}
i_1 & j_1 & k_1\\
\vdots & \vdots & \vdots \\
i_m & j_m & k_m
\end{bmatrix},
\]
respectively.

\subsection{Simplicial mappings}

A simplicial mapping on a triangular mesh $\M=\cup_{\ell=1}^m \tau_\ell$ is a particular piecewise affine mapping $f:\M\to\R^2$ solely determined by the mapping of vertices $f|_{\V(\M)}:\V(\M)\to\R^2$. More explicitly, suppose $f|_{\V(\M)}$ is given by $f(v_\ell) = \f_\ell\in\mathbb{R}^2$, for $\ell=1,\ldots,n$. For the point $\v\in\tau_\ell\in\F(\M)$, 
\[
f|_{\tau_\ell}(\v) =  \frac{1}{|\tau_\ell|} \left(|[\v,\v_{j_\ell},\v_{k_\ell}]| \, \f_{i_\ell} + |[\v_{i_\ell},\v,\v_{k_\ell}]| \, \f_{j_\ell} + |[\v_{i_\ell},\v_{j_\ell},\v]| \, \f_{k_\ell} \right).
\]
In practice, the function $f$ is represented as a real-valued matrix 
$$
\f = \begin{bmatrix}
\f_1^\top \\
\vdots\\
\f_n^\top
\end{bmatrix}
= \begin{bmatrix}
f_1^1 & f_1^2 \\
\vdots & \vdots\\
f_n^1 & f_n^2
\end{bmatrix}
:= \begin{bmatrix}
\f^1 & \f^2
\end{bmatrix}.
$$
In other words, a simplicial map on $\M$ is an embedding of $\M$ on the parametric domain $f(\M)$.

\section{Analysis of the stretch energy functional} \label{sec:SEM}

The stretch energy functional \cite{Yueh17,YuLW19} for simplicial surfaces is defined as
\begin{equation}\label{eq:E_S}
{E}_S({f}) = \frac{1}{2} \trace \left(\mathbf{f}^\top L_S({f}) \, \mathbf{f} \right)
= \frac{1}{2} \left( \mathbf{f}^{1\top} L_S({f}) \, \mathbf{f}^1 + \mathbf{f}^{2\top} L_S({f}) \, \mathbf{f}^2\right),
\end{equation}
where $L_S({f})$ is the weighted Laplacian matrix given by
\begin{equation} \label{eq:L_S}
{[L_S(f)]}_{i,j} =
   \begin{cases}
   \displaystyle
   -\frac{1}{2}\sum_{[v_i,v_j,v_k]\in\F(\M)} \frac{\cot(\theta_{i,j}^k(f))  \, |f([v_i,v_j,v_k])|}{ |[v_i,v_j,v_k]|}  &\mbox{if $[{v}_i,{v}_j]\in\mathcal{E}(\mathcal{M})$,}\\
   \displaystyle
   -\sum_{\ell\neq i} [L_S(f)]_{i,\ell} &\mbox{if $j = i$,}\\
   0 &\mbox{otherwise}
   \end{cases}
\end{equation}
in which $\theta_{i,j}^k(f)$ is the angle opposite to the edge $[\f_i,\f_j]$ at the point $\f_k$ on $f(\mathcal{M})$, as illustrated in Figure \ref{fig:cotWeight}. 
The stretch energy functional is defined in \cite{Yueh17,YuLW19} by imposing the area-preserving condition to the cotangent weights in the discrete Dirichlet energy.

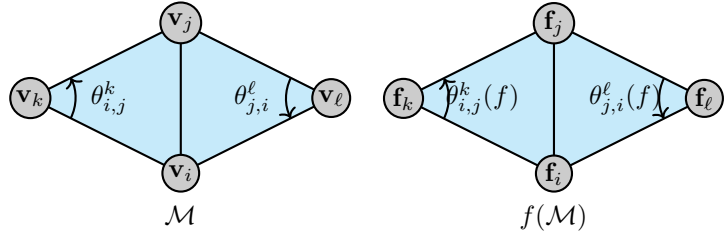
\begin{figure}
\center
\begin{tabular}{cc}
\begin{tikzpicture}[thick,scale=1]
\coordinate (v_i) at (0,0);
\coordinate (v_j) at (0,2);
\coordinate (v_k) at (2,1);
\coordinate (v_l) at (-2,1);
\filldraw[cyan!20] (v_i) -- (v_j) -- (v_k);
\filldraw[cyan!20] (v_i) -- (v_j) -- (v_l);
\pic[draw, ->, "$\theta_{i,j}^k$", angle eccentricity=1.75, angle radius=0.6cm]{angle = v_i--v_l--v_j};
\pic[draw, ->, "$\theta_{j,i}^\ell$", angle eccentricity=1.75, angle radius=0.6cm]{angle = v_j--v_k--v_i};
\draw{
(v_i) -- (v_j) -- (v_k) -- (v_i) -- (v_l) -- (v_j)
};
\tikzstyle{every node}=[circle, draw, fill=black!20,
                        inner sep=1pt, minimum width=2pt]
\draw{
(0,0) node{${\v}_i$}
(0,2) node{${\v}_j$}
(2,1) node{${\v}_\ell$}
(-2,1) node{${\v}_k$}
};
\end{tikzpicture}
&
\begin{tikzpicture}[thick,scale=1]
\coordinate (v_i) at (0,0);
\coordinate (v_j) at (0,2);
\coordinate (v_k) at (2,1);
\coordinate (v_l) at (-2,1);
\filldraw[cyan!20] (v_i) -- (v_j) -- (v_k);
\filldraw[cyan!20] (v_i) -- (v_j) -- (v_l);
\pic[draw, ->, "$\theta_{i,j}^k(f)$", angle eccentricity=1.75, angle radius=0.6cm]{angle = v_i--v_l--v_j};
\pic[draw, ->, "$\theta_{j,i}^\ell(f)$", angle eccentricity=1.75, angle radius=0.6cm]{angle = v_j--v_k--v_i};
\draw{
(v_i) -- (v_j) -- (v_k) -- (v_i) -- (v_l) -- (v_j)
};
\tikzstyle{every node}=[circle, draw, fill=black!20,
                        inner sep=1pt, minimum width=2pt]
\draw{
(0,0) node{$\f_i$}
(0,2) node{$\f_j$}
(2,1) node{$\f_\ell$}
(-2,1) node{$\f_k$}
};
\end{tikzpicture} \\
$\M$ & $f(\M)$
\end{tabular}
\caption{Illustrations for the angles in the cotangent weights of the Laplacian matrices \eqref{eq:L_D} and \eqref{eq:L_S}.}
\label{fig:cotWeight}
\end{figure}

\subsection{Gradient of the stretch energy}
We now derive a simplified formulation for the gradient of the stretch energy \eqref{eq:E_S}. 
Noting that the Laplacian matrix $L_S(f)$ in \eqref{eq:L_S} is dependent on $f$ so that the gradient 
$$
\nabla_{\f} E_S(f) := \begin{bmatrix}
\nabla_{\f^1} E_S(f) \\
\nabla_{\f^2} E_S(f)
\end{bmatrix}
$$
with
{\footnotesize
\begin{equation} \label{eq:GradE_S}
[\nabla_{\f^s} E_S(f)]_\ell = \frac12 \left(\sum_{j=1}^n [L_S(f)]_{\ell,j} f_j^s + \sum_{i=1}^n f_i^s [L_S(f)]_{i,\ell} + \sum_{i=1}^n \sum_{j=1}^n f_i^1 f_j^1 \frac{\partial}{\partial f_\ell^s}[L_S(f)]_{i,j} + \sum_{i=1}^n \sum_{j=1}^n f_i^2 f_j^2 \frac{\partial}{\partial f_\ell^s}[L_S(f)]_{i,j} \right)
\end{equation}}
involves a 3D tensor $\frac{\partial}{\partial f_\ell^s}[L_S(f)]_{i,j}$, which would cause difficulty for sparse computation. 
Fortunately, we can prove in the following theorem that the third and fourth terms in the right-hand-side of \eqref{eq:GradE_S} can be written in a neat formulation that can be implemented easily.

\begin{theorem} \label{thm:1}
The gradient of the stretch energy functional \eqref{eq:E_S} is formulated as
\begin{equation} \label{eq:thm1}
\nabla_{\f} E_S(f):=
\begin{bmatrix}
\nabla_{\f^1} E_S(f) \\
\nabla_{\f^2} E_S(f) 
\end{bmatrix}
=
\begin{bmatrix}
L_S(f) \, \f^1 + \h^1 \\
L_S(f) \, \f^2 + \h^2
\end{bmatrix}
\end{equation}
with
\begin{align} \label{eq:thm1_h1}
\h^1_i &= \frac{1}{2} \sum_{\tau=[\v_i,\v_j,\v_k]\in N_{\mathcal{F}}(\v_i)} \frac{|f(\tau)|}{|\tau|} (f_j^2-f_k^2), 
\end{align}
and
\begin{align} \label{eq:thm1_h2}
\h^2_i &= \frac{1}{2} \sum_{\tau=[\v_i,\v_j,\v_k]\in N_{\mathcal{F}}(\v_i)} \frac{|f(\tau)|}{|\tau|} (f_k^1-f_j^1),
\end{align}
where $N_{\mathcal{F}}(\v_i) = \{\tau \mid \v_i\subset\tau\}$ denotes the set of neighboring triangular faces of $\v_i$. 
\end{theorem}

\begin{proof}
For each triangular face $\tau = [\v_i,\v_j,\v_k]\in\F(\M)$, we let
\begin{equation} \label{eq:omega}
\omega_{i,j,k}(f) = \frac{\cot(\theta_{i,j}^k(f))  \, |f([\v_i,\v_j,\v_k])|}{2|[\v_i,\v_j,\v_k]|}.
\end{equation}
Then, $L_S(f)$ is written as
\begin{equation*}
{[L_S(f)]}_{i,j} =
   \begin{cases}
   -\sum_{[v_i,v_j,v_k]\in\F(\M)} \omega_{i,j,k}(f)  &\mbox{if $[{v}_i,{v}_j]\in\mathcal{E}(\mathcal{M})$,}\\
   -\sum_{\ell\neq i} [L_S(f)]_{i,\ell} &\mbox{if $j = i$,}\\
   0 &\mbox{otherwise.}
   \end{cases}
\end{equation*}
Noting that $E_S$ can be written as
\begin{align*}
E_S(f) & = \frac12 \text{trace}\left(\f^\top L_S(f) \,\f\right) = \frac12 \sum_{s=1}^2 \sum_{i=1}^n\sum_{j=1}^n f_i^s [L_S(f)]_{i,j} f_j^s.
\end{align*}
The entries of gradient $\nabla_\f E_S(f) = ((\nabla_{\f^1} E_S(f))^\top, (\nabla_{\f^2} E_S(f))^\top)^\top$ can be written as
{\footnotesize
\begin{align}
[\nabla_{\f^s} E(f)]_\ell & = \frac12 \frac{\partial}{\partial f_\ell^s} \left( \sum_{i=1}^n\sum_{j=1}^n f_i^1 [L_S(f)]_{i,j} f_j^1 + \sum_{i=1}^n\sum_{j=1}^n f_i^2 [L_S(f)]_{i,j} f_j^2 \right) \nonumber\\
&= \frac12 \left(\sum_{j=1}^n [L_S(f)]_{\ell,j} f_j^s + \sum_{i=1}^n f_i^s [L_S(f)]_{i,\ell} +  \sum_{i=1}^n \sum_{j=1}^n f_i^1 f_j^1 \frac{\partial}{\partial f_\ell^s}[L_S(f)]_{i,j} + \sum_{i=1}^n \sum_{j=1}^n f_i^2 f_j^2 \frac{\partial}{\partial f_\ell^s}[L_S(f)]_{i,j} \right) \nonumber\\
&= \sum_{j=1}^n [L_S(f)]_{\ell,j} f_j^s + \frac{1}{2} \sum_{i=1}^n \sum_{j=1}^n f_i^1 f_j^1 \frac{\partial}{\partial f_\ell^s}[L_S(f)]_{i,j} + \frac{1}{2} \sum_{i=1}^n \sum_{j=1}^n f_i^2 f_j^2 \frac{\partial}{\partial f_\ell^s}[L_S(f)]_{i,j}, ~ \ell=1, \ldots, n, \label{eq:GradE_S_ell}
\end{align}}
for $s=1,2$. 
To further simplify $\frac{\partial}{\partial f_\ell^s}L_S(f)$ in \eqref{eq:GradE_S_ell}, we reformulate $\omega_{i,j,k}$ in \eqref{eq:omega} as
\begin{align*}
\omega_{i,j,k}(f) &= 
\frac{\cot(\theta_{i,j}^k(f))  \, |f([\v_i,\v_j,\v_k])|}{2|[\v_i,\v_j,\v_k]|} \\
&= \frac{\cos(\theta_{i,j}^k(f))  \, |[\f_i,\f_j,\f_k]|}{2\sin(\theta_{i,j}^k(f))|[\v_i,\v_j,\v_k]|} && \text{since $|f([\v_i,\v_j,\v_k])| = |[\f_i,\f_j,\f_k]|$} \\
&= \frac{|[\f_k,\f_i]| \, |[\f_k,\f_j]| \cos(\theta_{i,j}^k(f))  \, |[\f_i,\f_j,\f_k]|}{2|[\f_k,\f_i]| \, |[\f_k,\f_j]| \sin(\theta_{i,j}^k(f)) |[\v_i,\v_j,\v_k]|} \\
&= \frac{|[\f_k,\f_i]| \, |[\f_k,\f_j]| \cos(\theta_{i,j}^k(f))}{4|[\v_i,\v_j,\v_k]|} && \text{since $|[\f_i,\f_j,\f_k]|=\frac{1}{2}|[\f_k,\f_i]| \, |[\f_k,\f_j]| \sin(\theta_{i,j}^k(f))$}\\
&= \frac{1}{4|[\v_i,\v_j,\v_k]|} (\f_i-\f_k)^\top (\f_j-\f_k).
\end{align*}
From the definition of the stretch energy \eqref{eq:L_S}, each triangular face $\tau = [\v_i, \v_j, \v_k]$ contributes a $3\times 3$ partial submatrix
{\footnotesize
\begin{equation} \label{eq:L_tau}
L_{\tau}(f) = L_{[\v_i,\v_j,\v_k]}(f) = \frac{1}{4|\tau|}
\begin{bmatrix}
\begin{array}{c}
(\f_i-\f_k)^\top(\f_j-\f_k)\\
~~ +(\f_i-\f_j)^\top(\f_k-\f_j)
\end{array} & -(\f_i-\f_k)^\top(\f_j-\f_k) & - (\f_i-\f_j)^\top(\f_k-\f_j) \\
-(\f_i-\f_k)^\top(\f_j-\f_k) & \begin{array}{c}
(\f_j-\f_i)^\top(\f_k-\f_i) \\
~~ +(\f_i-\f_k)^\top(\f_j-\f_k)
\end{array} & -(\f_i-\f_j)^\top(\f_i-\f_k) \\
-(\f_i-\f_j)^\top(\f_k-\f_j) & -(\f_j-\f_i)^\top(\f_k-\f_i) & 
\begin{array}{c}
(\f_i-\f_j)^\top(\f_i-\f_k) \\
~~ +(\f_i-\f_j)^\top(\f_k-\f_j) 
\end{array}
\end{bmatrix}.
\end{equation}}
By taking the partial derivatives of \eqref{eq:L_tau} with respect to $f_i^s$, $f_j^s$ and $f_k^s$, respectively, we obtain 
\begin{align}
\frac{\partial}{\partial f_i^s}L_{\tau}(f) &= \frac{1}{4|\tau|}
\begin{bmatrix}
      0 &        f_k^s - f_j^s &        f_j^s - f_k^s \\
f_k^s - f_j^s &    2f_i^s - 2f_k^s & f_j^s - 2f_i^s + f_k^s \\
f_j^s - f_k^s & f_j^s - 2f_i^s + f_k^s &    2f_i^s - 2f_j^s
\end{bmatrix}, \label{eq:L_i}\\
\frac{\partial}{\partial f_j^s}L_{\tau}(f) &= \frac{1}{4|\tau|}
\begin{bmatrix}
     2f_j^s - 2f_k^s & f_k^s - f_i^s & f_i^s - 2f_j^s + f_k^s \\
       f_k^s - f_i^s &       0   &        f_i^s - f_k^s \\
f_i^s - 2f_j^s + f_k^s & f_i^s - f_k^s &    2f_j^s - 2f_i^s
\end{bmatrix}, \label{eq:L_j}\\
\frac{\partial}{\partial f_k^s}L_{\tau}(f) &= \frac{1}{4|\tau|}
\begin{bmatrix}
     2f_k^s - 2f_j^s & f_i^s + f_j^s - 2f_k^s & f_j^s - f_i^s \\
f_i^s + f_j^s - 2f_k^s &      2f_k^s - 2f_i^s & f_i^s - f_j^s \\
       f_j^s - f_i^s &        f_i^s - f_j^s &         0
\end{bmatrix}. \label{eq:L_k}
\end{align}
Let $\mathbf{f}_\tau^s = \mathbf{f}_{[\v_i,\v_j,\v_k]}^s := (f_i^s, f_j^s, f_k^s)^\top$. 
With the formulation of \eqref{eq:L_i}--\eqref{eq:L_k}, it can be verify by direct computation that 
\begin{align} 
\mathbf{f}_\tau^{s\top} \frac{\partial}{\partial f_i^s}L_\tau(f) \, \mathbf{f}_\tau^s &= 0, \label{eq:fLf_i} \\ 
\mathbf{f}_\tau^{s\top} \frac{\partial}{\partial f_j^s}L_\tau(f) \, \mathbf{f}_\tau^s &= 0, \label{eq:fLf_j} \\  
\mathbf{f}_\tau^{s\top} \frac{\partial}{\partial f_k^s}L_\tau(f) \, \mathbf{f}_\tau^s &= 0, \label{eq:fLf_k} \\
\mathbf{f}_\tau^{s\top} \frac{\partial}{\partial f_\ell^s}L_\tau(f) \, \mathbf{f}_\tau^s &= 0, ~ \ell\notin\{i,j,k\}, \label{eq:fLf_ell}
\end{align}
for $s = 1,2$, and
{\footnotesize
\begin{align}
\mathbf{f}_\tau^{t\top} \frac{\partial}{\partial f_i^s}L_\tau(f) \, \mathbf{f}_\tau^t 
&= \frac{1}{2|\tau|} (f_j^t-f_k^t) (f_i^s f_j^t - f_j^s f_i^t + f_k^s f_i^t - f_i^s f_k^t + f_j^s f_k^t - f_k^s f_j^t)
= \frac{|f(\tau)|}{|\tau|} (-1)^t (f_j^t-f_k^t), \label{eq:gLg_i} \\ 
\mathbf{f}_\tau^{t\top} \frac{\partial}{\partial f_j^s}L_\tau(f) \, \mathbf{f}_\tau^t 
&= \frac{1}{2|\tau|} (f_k^t-f_i^t) (f_i^s f_j^t - f_j^s f_i^t + f_k^s f_i^t - f_i^s f_k^t + f_j^s f_k^t - f_k^s f_j^t)
= \frac{|f(\tau)|}{|\tau|} (-1)^t (f_k^t-f_i^t), \label{eq:gLg_j} \\ 
\mathbf{f}_\tau^{t\top} \frac{\partial}{\partial f_k^s}L_\tau(f) \, \mathbf{f}_\tau^t 
&= \frac{1}{2|\tau|} (f_i^t-f_j^t) (f_i^s f_j^t - f_j^s f_i^t + f_k^s f_i^t - f_i^s f_k^t + f_j^s f_k^t - f_k^s f_j^t)
= \frac{|f(\tau)|}{|\tau|} (-1)^t (f_i^t-f_j^t), \label{eq:gLg_k} \\
\mathbf{f}_\tau^{t\top} \frac{\partial}{\partial f_\ell^s}L_\tau(f) \, \mathbf{f}_\tau^t &= 0, ~ \ell\notin\{i,j,k\}, \label{eq:gLg_ell}
\end{align}}
for $(s,t)\in\{(1,2),(2,1)\}$.

Noting that the matrix $L_S(f)$ is assembled by $\{L_\tau(f)\}_{\tau\in\mathcal{F}(\mathcal{M})}$, therefore, \eqref{eq:thm1}--\eqref{eq:thm1_h2} hold by \eqref{eq:GradE_S_ell} and \eqref{eq:fLf_i}--\eqref{eq:gLg_ell}. 
\end{proof}

The formulas of $\h^1$ and $\h^2$ in \eqref{eq:thm1_h1} and \eqref{eq:thm1_h2} look pretty and familiar. Indeed, they are exactly the same as $L_S(f)\, \f^1$ and $L_S(f)\, \f^2$, respectively, which is stated in the following corollary. 

\begin{corollary} \label{cor:1}
The vectors $\h^1$ and $\h^2$ in \eqref{eq:thm1_h1} and \eqref{eq:thm1_h2} satisfy
$$
\h^1 = L_S(f)\, \f^1 
~\text{ and }~
\h^2 = L_S(f)\, \f^2, 
$$
respectively. As a result, the gradient of $E_S$ can also be formulated as
\begin{equation} \label{eq:cor1}
\nabla_{\f} E_S(f):=
\begin{bmatrix}
\nabla_{\f^1} E_S(f) \\
\nabla_{\f^2} E_S(f) 
\end{bmatrix}
=
2\begin{bmatrix}
L_S(f) \, \f^1 \\
L_S(f) \, \f^2
\end{bmatrix}.
\end{equation}
\end{corollary}
\begin{proof}
Let $\tau=[\v_i,\v_j,\v_k]\in\F(\M)$. 
By multiplying $L_{\tau}(f)$ in \eqref{eq:L_tau} with $\f_{\tau}^s:=(f_i^s, f_j^s, f_k^s)^\top$, $s=1,2$, we obtain 
$$
L_{\tau}(f) \, \f_{\tau}^1 = \frac{1}{4|\tau|} \begin{bmatrix}
(f_j^2-f_k^2) (f_i^1 f_j^2 - f_j^1 f_i^2 + f_k^1 f_i^2 - f_i^1 f_k^2 + f_j^1 f_k^2 - f_k^1 f_j^2) \\
(f_k^2-f_i^2) (f_i^1 f_j^2 - f_j^1 f_i^2 + f_k^1 f_i^2 - f_i^1 f_k^2 + f_j^1 f_k^2 - f_k^1 f_j^2) \\
(f_i^2-f_j^2) (f_i^1 f_j^2 - f_j^1 f_i^2 + f_k^1 f_i^2 - f_i^1 f_k^2 + f_j^1 f_k^2 - f_k^1 f_j^2) 
\end{bmatrix}
= \frac{|f(\tau)|}{2|\tau|} \begin{bmatrix}
f_j^2-f_k^2 \\
f_k^2-f_i^2\\
f_i^2-f_j^2 
\end{bmatrix}
$$
and
$$
L_{\tau}(f) \, \f_{\tau}^2 = \frac{1}{4|\tau|} \begin{bmatrix}
(f_k^1-f_j^1) (f_i^1 f_j^2 - f_j^1 f_i^2 + f_k^1 f_i^2 - f_i^1 f_k^2 + f_j^1 f_k^2 - f_k^1 f_j^2) \\
(f_i^1-f_k^1) (f_i^1 f_j^2 - f_j^1 f_i^2 + f_k^1 f_i^2 - f_i^1 f_k^2 + f_j^1 f_k^2 - f_k^1 f_j^2) \\
(f_j^1-f_i^1) (f_i^1 f_j^2 - f_j^1 f_i^2 + f_k^1 f_i^2 - f_i^1 f_k^2 + f_j^1 f_k^2 - f_k^1 f_j^2) 
\end{bmatrix}
= \frac{|f(\tau)|}{2|\tau|} \begin{bmatrix}
f_k^1-f_j^1 \\
f_i^1-f_k^1\\
f_j^1-f_i^1 
\end{bmatrix}.
$$
By summing the formulas over every $\tau\in\F(\M)$, we obtain
\begin{align*}
L_S(f) \,\f^1 &= \sum_{\tau=[\v_i,\v_j,\v_k]\in N_{\mathcal{F}}(\v_i)} \frac{|f(\tau)|}{2|\tau|} (f_j^2-f_k^2), \end{align*}
and
\begin{align*}
L_S(f) \,\f^2 &= \sum_{\tau=[\v_i,\v_j,\v_k]\in N_{\mathcal{F}}(\v_i)} \frac{|f(\tau)|}{2|\tau|} (f_k^1-f_j^1),
\end{align*}
which are exactly the same as the formulas of $\h^1$ and $\h^2$ in \eqref{eq:thm1_h1} and \eqref{eq:thm1_h2}.
\end{proof}

\subsection{Geometric interpretation of the stretch energy}

Next, we show in the following theorem that the stretch energy $E_S$ can be written as the sum of squares of the image area of triangular faces divided by the triangular face area, which is a geometric interpretation of the stretch energy. 

\begin{theorem}\label{thm:2}
The stretch energy defined in \eqref{eq:E_S} can be reformulated as
\begin{equation*} 
E_S(f) = \sum_{\tau\in\F(\M)} \frac{|f(\tau)|^2}{|\tau|}.
\end{equation*}
\end{theorem}
\begin{proof}
By the explicit formula of $L_\tau(f)$ in \eqref{eq:L_tau}, a direct derivation yields that
\begin{align*}
E_S(f) &= \frac{1}{2} \sum_{s=1}^2 \sum_{[\v_i,\v_j,\v_k]\in\mathcal{F}(\mathcal{M})} \begin{bmatrix} f_i^s & f_j^s & f_k^s \end{bmatrix} L_{[\v_i,\v_j,\v_k]}(f) \begin{bmatrix} f_i^s \\ f_j^s \\ f_k^s \end{bmatrix} \\
&= \sum_{[\v_i,\v_j,\v_k]\in\mathcal{F}(\mathcal{M})} \frac{1}{4|[\v_i,\v_j,\v_k]|} (f_i^1 f_j^2 - f_j^1 f_i^2 + f_k^1 f_i^2 - f_i^1 f_k^2 + f_j^1 f_k^2 - f_k^1 f_j^2)^2 \\
&= \sum_{\tau\in\F(\M)} \frac{|f(\tau)|^2}{|\tau|}.
\end{align*}
The last equality held by the identity
$$
|f([\v_i,\v_j,\v_k])| = \frac{1}{2}(f_i^1 f_j^2 - f_j^1 f_i^2 + f_k^1 f_i^2 - f_i^1 f_k^2 + f_j^1 f_k^2 - f_k^1 f_j^2).
$$
\end{proof}

In particular, when $f$ is an area-preserving mapping, the stretch energy would be the area of its image, as stated in the following corollary. 

\begin{corollary}
For an area-preserving mapping $f:M\to\R^2$, 
$$
E_S(f) = |f(\M)|.
$$
\end{corollary}
\begin{proof}
Since $f$ preserves the area, $|f(\tau)| = |\tau|$, for every $\tau\in\F(\M)$. 
By Theorem \ref{thm:2}, 
$$
E_S(f) = \sum_{\tau\in\F(\M)} |f(\tau)| = |f(\M)|.
$$
\end{proof}

\subsection{Minimization of the stretch energy}

We now provide the foundation of the stretch energy minimization for the computation of area-preserving mappings, which is stated in the following theorem. 

\begin{theorem} \label{thm:3}
Given a simplicial surface $\M$. Under the constraint that the total area of the surface remains unchanged, the minimal value of $E_S$ occurs only at area-preserving mappings, i.e., 
$$
f = \argmin_{|g(\M)|=|\M|} E_S(g)
$$
if and only if $|f(\tau)|=|\tau|$, for every $\tau\in\F(\M)$.
\end{theorem}
\begin{proof}
Without loss of generality, we suppose that the total area of surface the image are normalized to be $1$, i.e., 
$$
\sum_{k=1}^m |\tau_k| = |\M|=1 ~\text{ and }~ \sum_{k=1}^m |f(\tau_k)| = |f(\M)|=1.
$$
Let $y_k:= |\tau_k| \in(0,1)$ and $x_k:= |f(\tau_k)| \in(0,1)$, $k=1,\ldots,m$, be the area of $k$th triangular face and its image, respectively. 
Since $\sum_{k=1}^m x_k = 1$ and $\sum_{k=1}^m y_k = 1$, we write
\begin{equation} \label{eq:xy}
x_m = 1-\sum_{k=1}^{m-1} x_k ~\text{ and }~ 
y_m = 1-\sum_{k=1}^{m-1} y_k,
\end{equation}
respectively. 
By \eqref{eq:xy} together with Theorem \ref{thm:2}, it suffices to show
$$
E_S(x_1, \ldots, x_{m-1}) = \sum_{k=1}^{m-1} \frac{x_k^2}{y_k} + \frac{(1-\sum_{k=1}^{m-1} x_k)^2}{1-\sum_{k=1}^{m-1} y_k}\geq 1,
$$
and the equality holds if and only if $x_k=y_k$, for $k=1, \ldots, m$. 
Noting that the critical points of $E_S(x_1, \ldots, x_{m-1})$ satisfies
$$
\frac{\partial E_S(x_1, \ldots, x_{m-1})}{\partial x_k} = \frac{2x_k}{y_k} - \frac{2(1-\sum_{k=1}^{m-1} x_k)}{1-\sum_{k=1}^{m-1} y_k} = 0,
$$
for $k=1, \ldots, m-1$, i.e.,
\begin{equation} \label{eq:thm3_1}
\frac{2x_k(1-\sum_{k=1}^{m-1} y_k) - 2 y_k (1-\sum_{k=1}^{m-1} x_k) }{y_k(1-\sum_{k=1}^{m-1} y_k)} = 0.
\end{equation}
Since $y_k>0$ and $1-\sum_{k=1}^{m-1} y_k>0$, \eqref{eq:thm3_1} implies
$$
x_k\left(1-\sum_{k=1}^{m-1} y_k\right) =  y_k \left(1-\sum_{k=1}^{m-1} x_k\right).
$$
Without loss of generality, let
$$
\frac{x_k}{y_k} = \frac{1-\sum_{k=1}^{m-1} x_k}{1-\sum_{k=1}^{m-1} y_k} = \frac{x_m}{y_m} := c,
$$
for $k=1, \ldots, m-1$, where $c$ is a constant, 
i.e., $x_k = cy_k$, for $k=1, \ldots, m$. 
Since $\sum_{k=1}^m y_k=1$ and $\sum_{k=1}^m x_k=1$, 
$$
c = c \sum_{k=1}^m y_k = \sum_{k=1}^m x_k = 1.
$$
Therefore, $|f(\tau_k)| = x_k = y_k = |\tau_k|$, $k=1, \ldots, m$. 
\end{proof}

Theorem \ref{thm:3} provides a rigorous proof that the stretch energy is the right functional to minimize when computing area-preserving mappings. 
Noting that the stretch energy is always larger or equal to the image area, and the equality holds if and only if the mapping is area-preserving. 
It is natural to further define the authalic energy as
\begin{equation} \label{eq:E_A}
E_A(f) = E_S(f) - \mathcal{A}(f),
\end{equation}
where $\mathcal{A}(f)$ measures the area of the image of $f$. 
Then, a quick corollary follows from Theorem \ref{thm:3}. 

\begin{corollary} \label{cor:3}
The authalic energy \eqref{eq:E_A} satisfies $E_A(f) \geq 0$ and the equality holds if and only if $f$ is area-preserving.
\end{corollary}

Based on Corollary \ref{cor:3}, the authalic energy \eqref{eq:E_A} can be used to measure how far a mapping is from area-preserving.

\section{Stretch energy minimization for square-shaped authalic mappings}
\label{sec:SEMAlg}

In this section, we demonstrate the stretch energy minimization algorithm for the computation of square-shaped area-preserving mappings of simply connected open simplicial surfaces. 
First, we introduce the square-shaped boundary constraint in Subsection \ref{subsec:boundary}. 
Then, we compute a harmonic mapping as the initial mapping, which is introduced in Subsection \ref{subsec:initial}. 
Finally, we introduce the iterative procedure for the stretch energy minimization in Subsection \ref{subset:SEM}. 

\subsection{Square-shaped boundary constraints} \label{subsec:boundary}

We denote the index sets of boundary vertices sorted in counterclockwise order and the interior vertices as
$$
\B = \{b \mid v_b\in\partial\M \}
~\text{ and }~
\I = \{i\mid v_i\notin\partial\M \},
$$
respectively. 
Suppose the corner points $\v_{\C_1}, \ldots, \v_{\C_4}$ are selected and mapped to $(0,0)$, $(1,0)$, $(1,1)$, and $(0,1)$, respectively. 
Based on the selected corner indices $\C = \{\C_1, \C_2, \C_3, \C_4\}$, the boundary indices $\B$ is classified into $4$ categories
$$
\Y_0=\{\C_1, \ldots, \C_2\}, ~
\X_1=\{\C_2, \ldots, \C_3\}, ~
\Y_1=\{\C_3, \ldots, \C_4\}, ~ \text{ and } ~ 
\X_0=\{\C_4, \ldots, \C_1\}. 
$$
Each category is a subset of $\B$ sorted in counterclockwise order. 
The square-shaped constraint is then formulated as
$$
\f^1_{\X_0} = \0, ~ 
\f^1_{\X_1} = \1, ~
\f^2_{\Y_0} = \0, ~ \text{ and } ~
\f^2_{\Y_1} = \1.
$$
We denote $\I_s$ and $\B_s$ as the sets of indices of unknown entries and constrained entries of $\f^s$, respectively, for $s=1,2$. Then, 
\begin{equation} \label{eq:IB}
\I_1 = \I \cup \Y_0 \cup \Y_1, ~
\B_1 = \X_0 \cup \X_1, ~
\I_2 = \I \cup \X_0 \cup \X_1, ~ \text{ and } ~
\B_2 = \Y_0 \cup \Y_1. 
\end{equation}

\subsection{Initial mapping}
\label{subsec:initial}

The initial mapping is chosen to be a harmonic mapping $f:\mathcal{M}\to\mathbb{R}^2$ computed by minimizing the Dirichlet energy functional \cite{Hutc91,DeMA02} defined as
$$
E_D(f) = \frac{1}{2} \left( \mathbf{f}^{1\top} L_D \, \mathbf{f}^1 + \mathbf{f}^{2\top} L_D \, \mathbf{f}^2 \right),
$$
with $L_D$ being the Laplacian matrix
\begin{equation} \label{eq:L_D}
{[L_D]}_{i,j} =
   \begin{cases}
   \displaystyle
   -\frac{1}{2}\sum_{[v_i,v_j,v_k]\in\F(\M)} \cot\theta_{i,j}^k  &\mbox{if $[{\v}_i,{\v}_j]\in\mathcal{E}(\mathcal{M})$,}\\
   \displaystyle
   -\sum_{\ell\neq i} [L_D]_{i,\ell} &\mbox{if $j = i$,}\\
   0 &\mbox{otherwise}
   \end{cases}
\end{equation}
in which $\theta_{i,j}^k$ is the angle opposite to the edge $[{v}_i,{v}_j]$ at the vertex ${v}_k$ on $\mathcal{M}$, as illustrated in Figure \ref{fig:cotWeight}. 
A minimizer $f$ of $E_D$ is called a harmonic mapping that satisfies
$$
\nabla_{\f} E_D(f) :=
\begin{bmatrix}
\nabla_{\f^1} E_D(f)\\
\nabla_{\f^2} E_D(f)
\end{bmatrix}
=
\begin{bmatrix}
L_D \, \f^1\\
L_D \, \f^2  
\end{bmatrix}
= \0, 
$$
for $s=1,2$. 
Under the square-shaped boundary constraints, the equations for the minimization of $E_D$ are formulated as the linear systems
$$
[L_D]_{\I_s,\I_s} \f^s_{\I_s} = -[L_D]_{\I_s,\B_s} \f^s_{\B_s},
$$
for $s=1,2$, where $\I_s$ and $\B_s$ are given in \eqref{eq:IB}.

\subsection{Iterations for the stretch energy minimization} \label{subset:SEM}

With the gradient formula \eqref{eq:cor1} in Corollary \ref{cor:1}, the critical points of $E_S$ would satisfy
$$
\nabla_{\f^s} E_S(f) = 2 L_S(f) \, \f^s = \0,
$$
for $s=1,2$. 
Under the square-shaped boundary constraints, the minimization can be achieved by the fixed-point iterations
$$
[L_S(f^{(n)})]_{\I_s,\I_s} \f^{s(n+1)}_{\I_s} = -[L_S(f^{(n)})]_{\I_s,\B_s} \f^{s(n)}_{\B_s},
$$
for $s=1,2$, where $\I_s$ and $\B_s$ are given in \eqref{eq:IB}.

The computational procedure in detail of the stretch energy minimization for area-preserving mappings is summarized in Algorithm \ref{alg:SEM}. 

\begin{algorithm}[]
\caption{Stretch energy minimization for square-shaped area-preserving mappings}
\label{alg:SEM}
\begin{algorithmic}[1]
\Require A simply connected open surface $\M$, indices of corner vertices $\C_1, \ldots, \C_4$. 
\Ensure A square-shaped area-preserving mapping $\f$.
\State Let $\B = \{b \mid v_b\in\partial\M \}$ be the index set of boundary vertices sorted in counterclockwise order. 
\State Let $\I = \{i\mid v_i\notin\partial\M \}$ be the index set of interior vertices.  
\State Classify indices of boundary vertices $\B$ into 
\[
\Y_0=\{\C_1, \ldots, \C_2\}, \X_1=\{\C_2, \ldots, \C_3\}, \Y_1=\{\C_3, \ldots, \C_4\}, \text{ and } \X_0=\{\C_4, \ldots, \C_1\}. 
\]
\State Let $\f^1_{\X_0} = \0$, $\f^1_{\X_1} = \1$, $\f^2_{\Y_0} = \0$, and $\f^2_{\Y_1} = \1$. 
\State Let $\I_1 = \I \cup \Y_0 \cup \Y_1$,  $\B_1 = \X_0 \cup \X_1$, $\I_2 = \I \cup \X_0 \cup \X_1$, and $\B_2 = \Y_0 \cup \Y_1$. 
\State Let $L=L_D$ be the Laplacian matrix defined as \eqref{eq:L_D}. 
\State Solve the linear systems $L_{\I_1,\I_1} \f^1_{\I_1} = -L_{\I_1,\B_1} \f^1_{\B_1}$ and $L_{\I_2,\I_2} \f^2_{\I_2} = -L_{\I_2,\B_2} \f^2_{\B_2}$, respectively. \label{alg:SEM_7}
\While{not convergent}
\State Let $L=L_S(f)$ be the Laplacian matrix defined as \eqref{eq:L_S}. 
\State Solve the linear systems $L_{\I_1,\I_1} \f^1_{\I_1} = -L_{\I_1,\B_1} \f^1_{\B_1}$ and $L_{\I_2,\I_2} \f^2_{\I_2} = -L_{\I_2,\B_2} \f^2_{\B_2}$, respectively.
\EndWhile
\State \Return the mapping $\f = [\f^1, \f^2]$.
\end{algorithmic}
\end{algorithm}

\section{Numerical experiments}
\label{sec:Numerical}

Now, we demonstrate the numerical results of the square-shaped area-preserving mappings computed by the stretch energy minimization Algorithm \ref{alg:SEM}. All the experiments are performed in MATLAB on MacBook Pro M1 with 16 GB RAM. The benchmark triangular mesh models, shown in Figure \ref{fig:Model} are obtained from the AIM@SHAPE shape repository \cite{AIM}, the Stanford 3D scanning repository \cite{Stanford}, and Sketchfab \cite{Sketchfab}. 
The areas of the mesh models are normalized to be $1$ so that there would be no global scaling when the target domain is selected to be a unit square. Some of the mesh models are resampled or modified so that every triangular face contains at least one interior vertices.

The initial mappings, shown in Figure \ref{fig:harmonic}, of the stretch energy minimization Algorithm \ref{alg:SEM} is a square-shaped harmonic mapping introduced in Subsection \ref{subsec:initial}. 
The resulting area-preserving mappings of the benchmark mesh models computed by the stretch energy minimization Algorithm \ref{alg:SEM} are demonstrated in Figure \ref{fig:SEM}. We see that the area-preserving mappings look very different from the initial harmonic mappings.

\begin{figure}
\centering
\begin{tabular}{cccc}
Chinese Lion & Femur & Max Planck & Left Hand\\
\includegraphics[height=4cm]{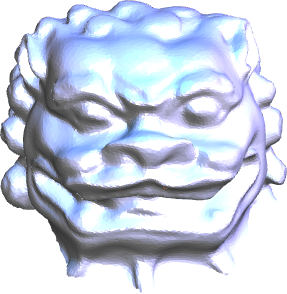} &
\includegraphics[height=4cm]{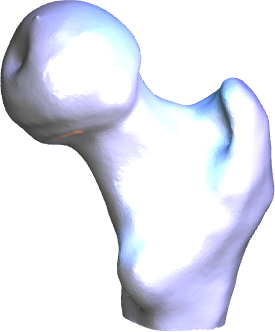} &
\includegraphics[height=4cm]{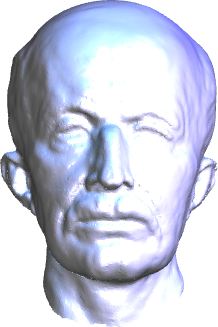} & 
\includegraphics[height=4cm]{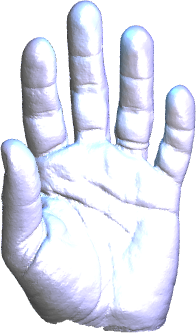}
\\[0.25cm]
Knit Cap Man & Bimba Statue & Buddha & Nefertiti Statue\\
\includegraphics[height=4cm]{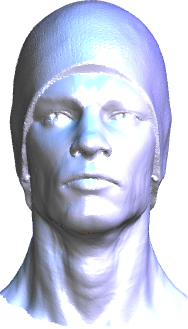} &
\includegraphics[height=4cm]{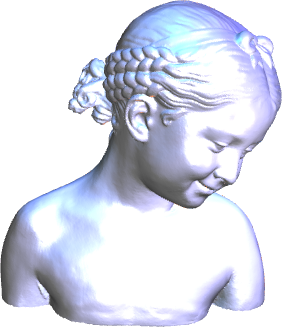} &
\includegraphics[height=4cm]{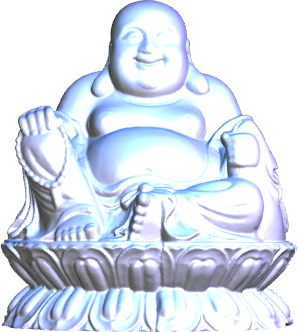} &
\includegraphics[height=4cm]{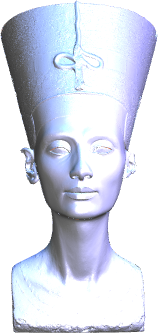} 
\end{tabular}
\caption{The benchmark triangular mesh models.}
\label{fig:Model}
\end{figure}

\begin{figure}
\centering
\begin{tabular}{cccc}
Chinese Lion & Femur & Max Planck & Left Hand\\
\includegraphics[height=3.5cm]{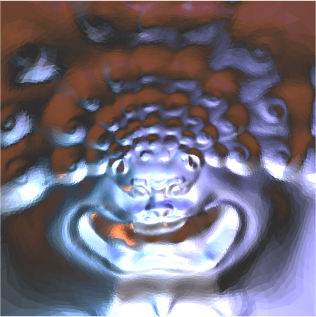} &
\includegraphics[height=3.5cm]{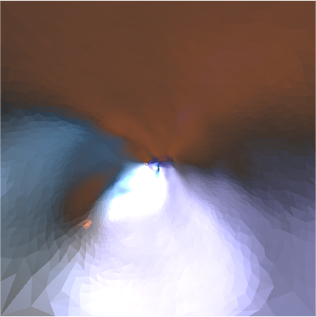} &
\includegraphics[height=3.5cm]{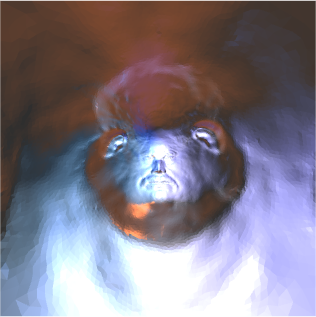} & 
\includegraphics[height=3.5cm]{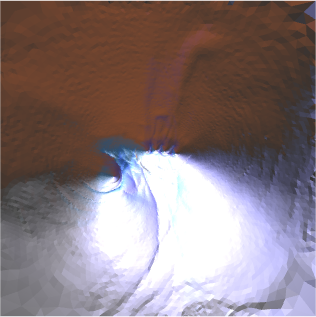}
\\[0.25cm]
Knit Cap Man & Bimba Statue & Buddha & Nefertiti Statue\\
\includegraphics[height=3.5cm]{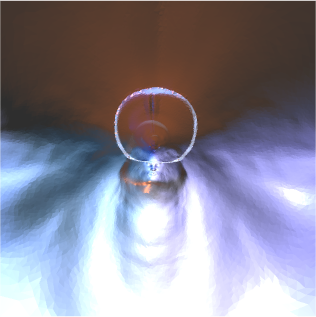} &
\includegraphics[height=3.5cm]{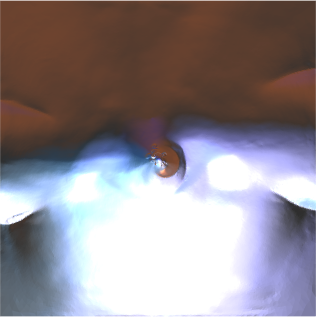} &
\includegraphics[height=3.5cm]{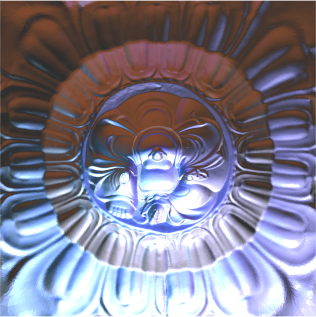} &
\includegraphics[height=3.5cm]{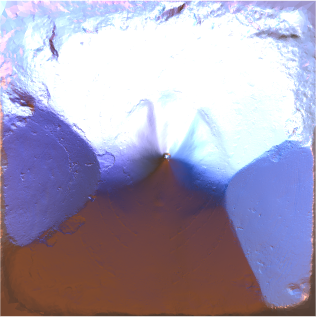} \end{tabular}
\caption{The square-shaped initial harmonic mappings of benchmark triangular mesh models computed by Step \ref{alg:SEM_7} of Algorithm \ref{alg:SEM}.}
\label{fig:harmonic}
\end{figure}

\begin{figure}
\centering
\begin{tabular}{cccc}
Chinese Lion & Femur & Max Planck & Left Hand\\
\includegraphics[height=3.5cm]{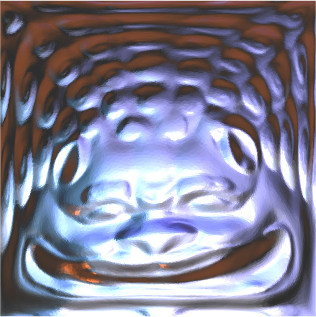} &
\includegraphics[height=3.5cm]{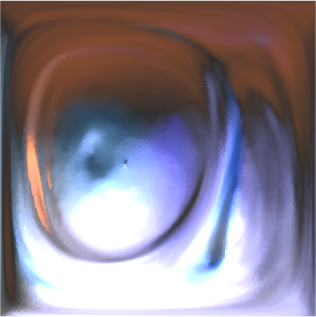} &
\includegraphics[height=3.5cm]{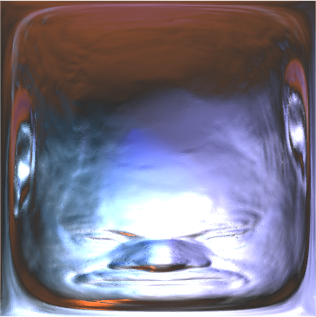} & 
\includegraphics[height=3.5cm]{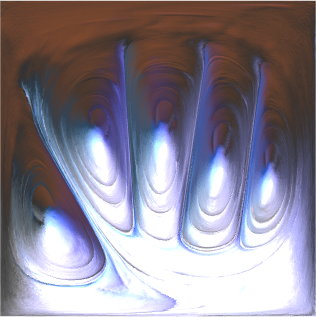}
\\[0.25cm]
Knit Cap Man & Bimba Statue & Buddha & Nefertiti Statue\\
\includegraphics[height=3.5cm]{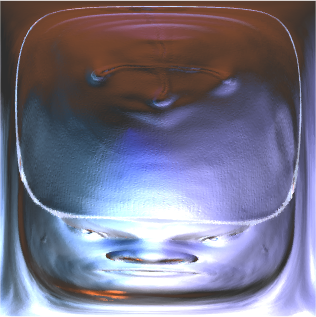} &
\includegraphics[height=3.5cm]{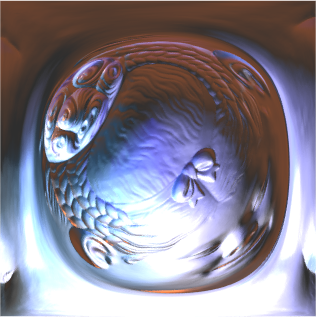} &
\includegraphics[height=3.5cm]{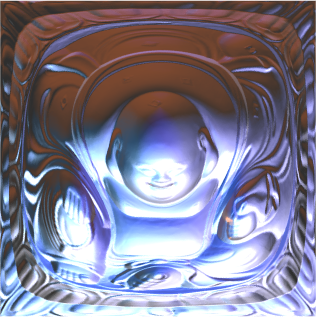} &
\includegraphics[height=3.5cm]{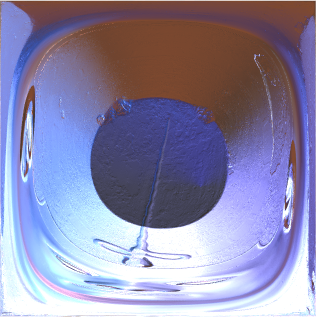} \end{tabular}
\caption{The square-shaped area-preserving mappings of benchmark triangular mesh models computed by the stretch energy minimization Algorithm \ref{alg:SEM}.}
\label{fig:SEM}
\end{figure}

Let $f:\M\to\R^2$ be the resulting area-preserving mapping computed by Algorithm \ref{alg:SEM}. 
Noting that the image $|f(\M)|$ is a unit square with the area being $1$. As stated in Corollary \ref{cor:3}, the authalic energy $E_A$ defined in \eqref{eq:E_A} satisfies
$$
E_A(f) = E_S(f) - 1,
$$
and $E_A(f)=0$ if and only if $f$ is area-preserving. 
The decrease of the authalic energy $E_A$ is equivalent to the decrease of the stretch energy $E_S$. 
To verify the performance of the stretch energy minimization Algorithm \ref{alg:SEM} on decreasing the stretch energy $E_S$ defined in \eqref{eq:E_S}, in Figure \ref{fig:Ea}, we demonstrate the relationship between the authalic energy $E_A$ and the number of iterations of Algorithm \ref{alg:SEM}. 
We observe that the authalic energy is decreased drastically to nearly zero at the first three iteration steps, which indicates that Algorithm \ref{alg:SEM} performs effectively on decreasing the authalic energy.

\begin{figure}
\centering
\begin{tabular}{cccc}
Chinese Lion & Femur & Max Planck & Left Hand\\
\includegraphics[height=4.3cm]{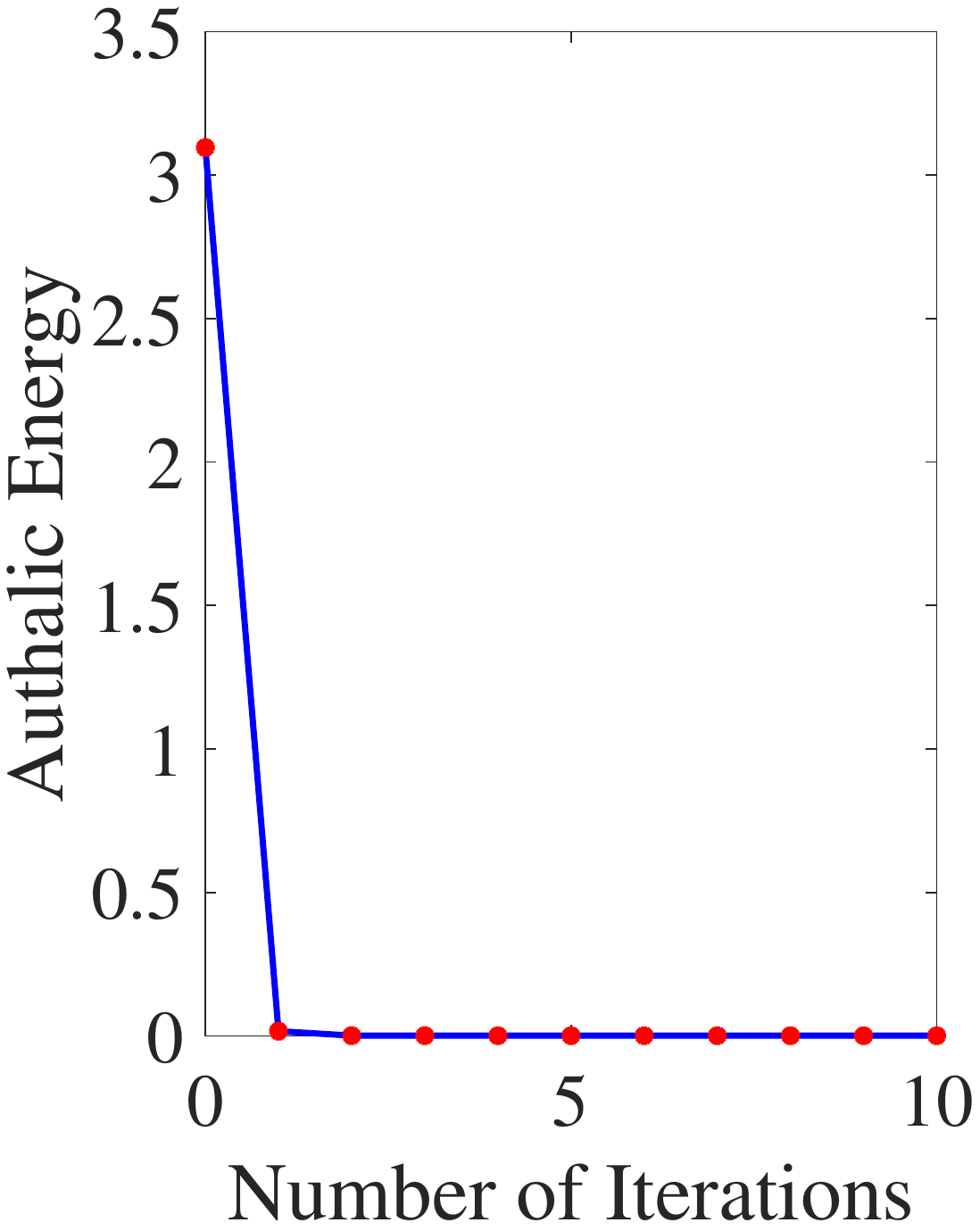} & 
\includegraphics[height=4.3cm]{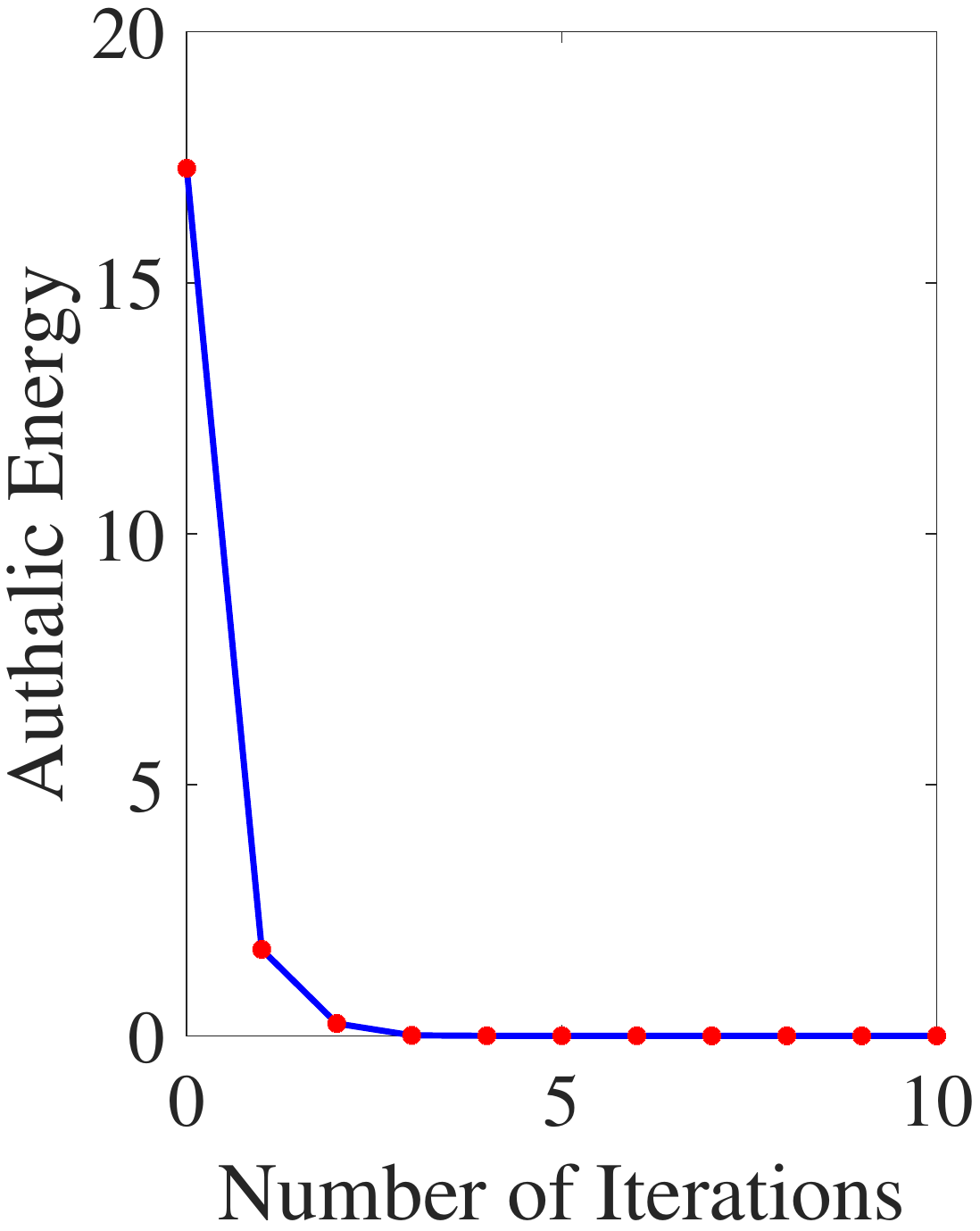} & 
\includegraphics[height=4.3cm]{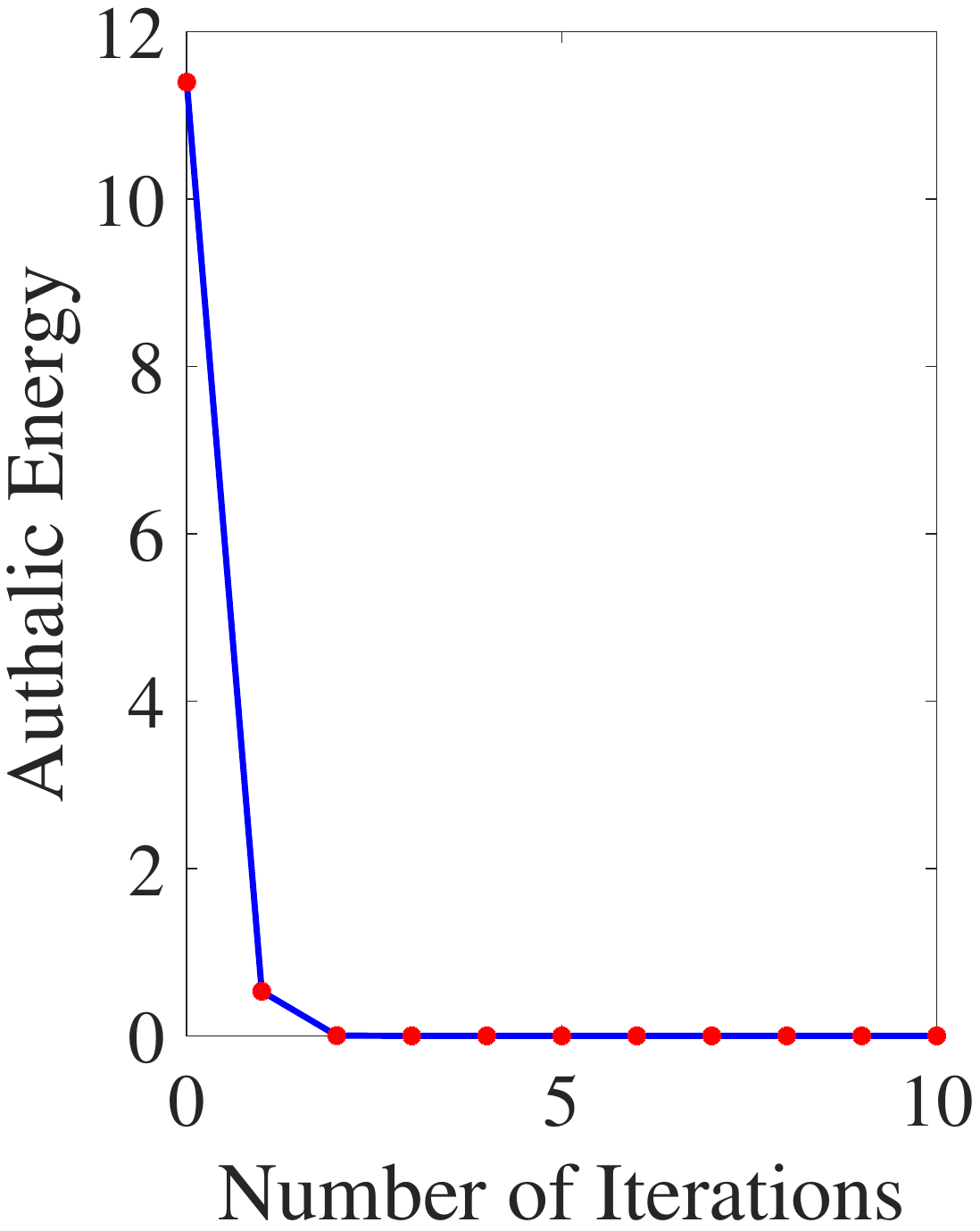} &
\includegraphics[height=4.3cm]{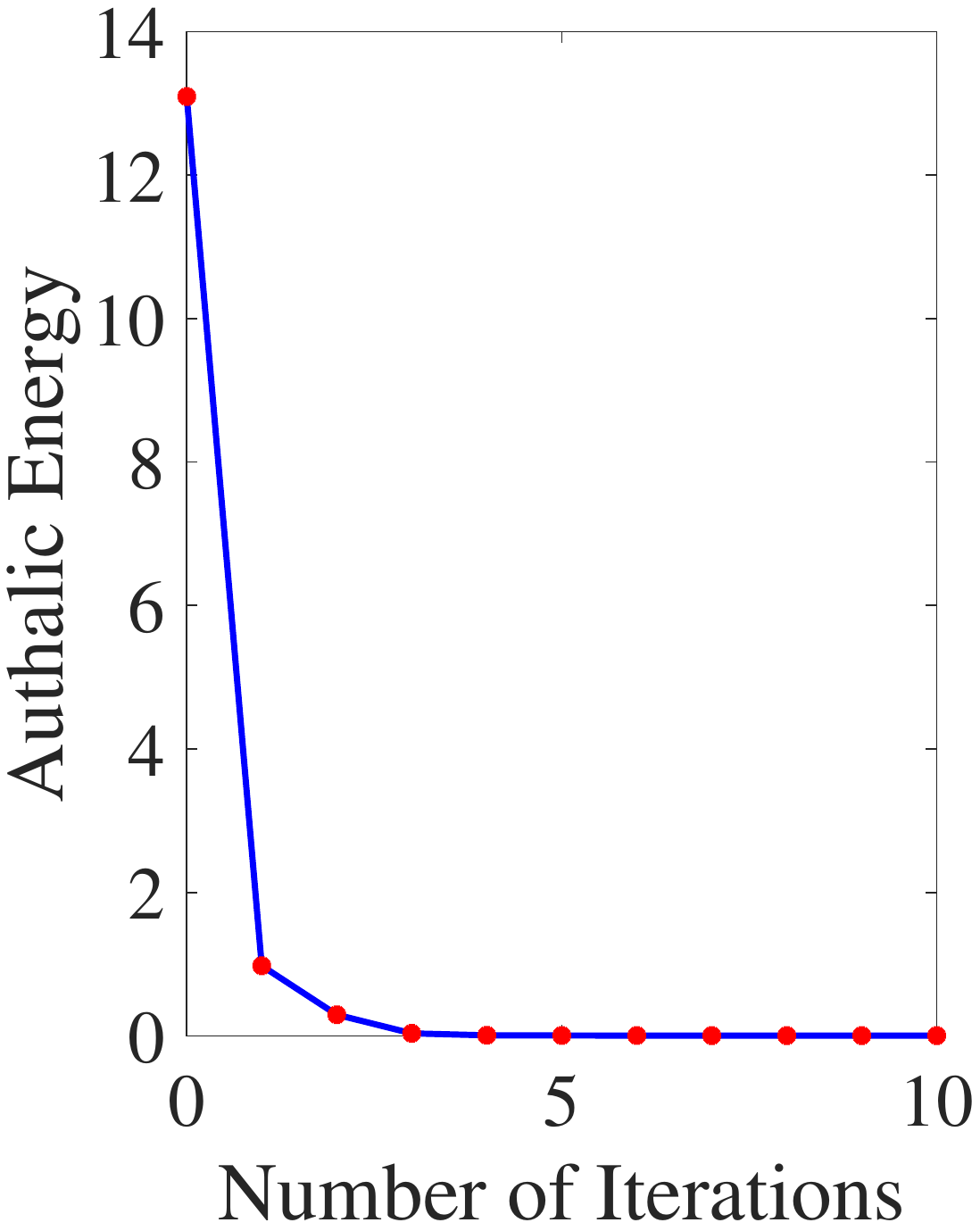} 
\\[0.5cm]
Knit Cap Man  & Bimba Statue & Buddha & Nefertiti Statue \\
\includegraphics[height=4.3cm]{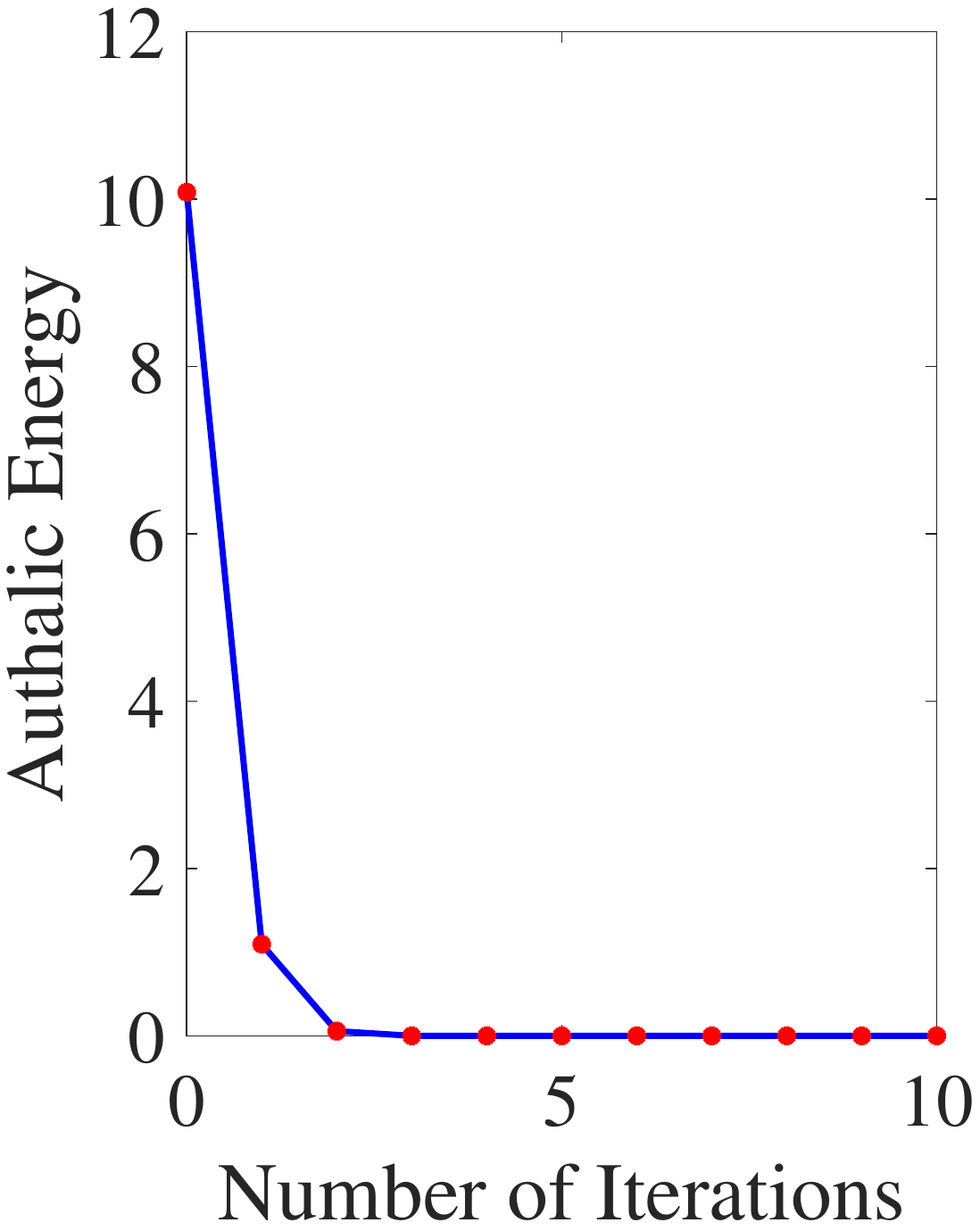} &
\includegraphics[height=4.3cm]{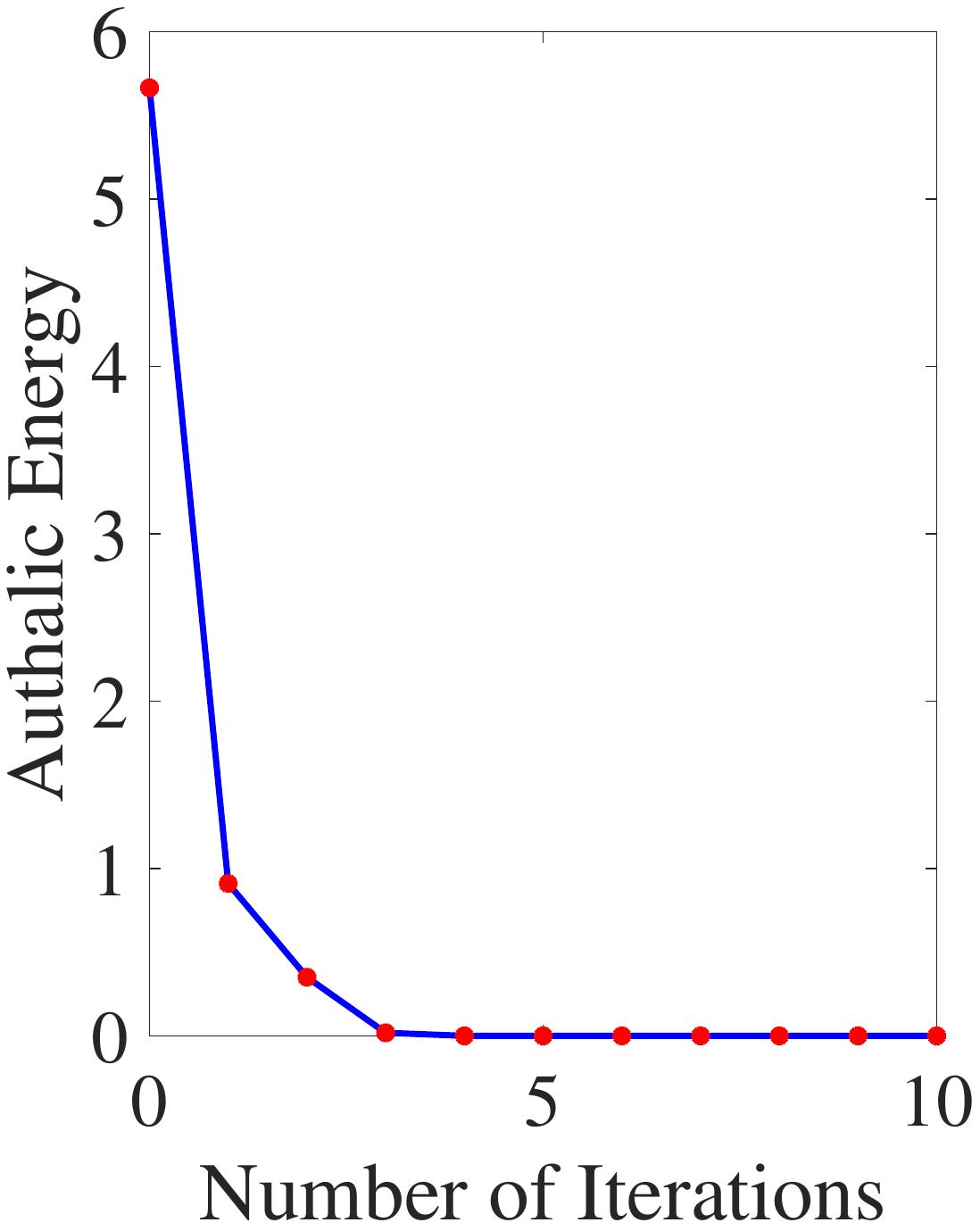} &
\includegraphics[height=4.3cm]{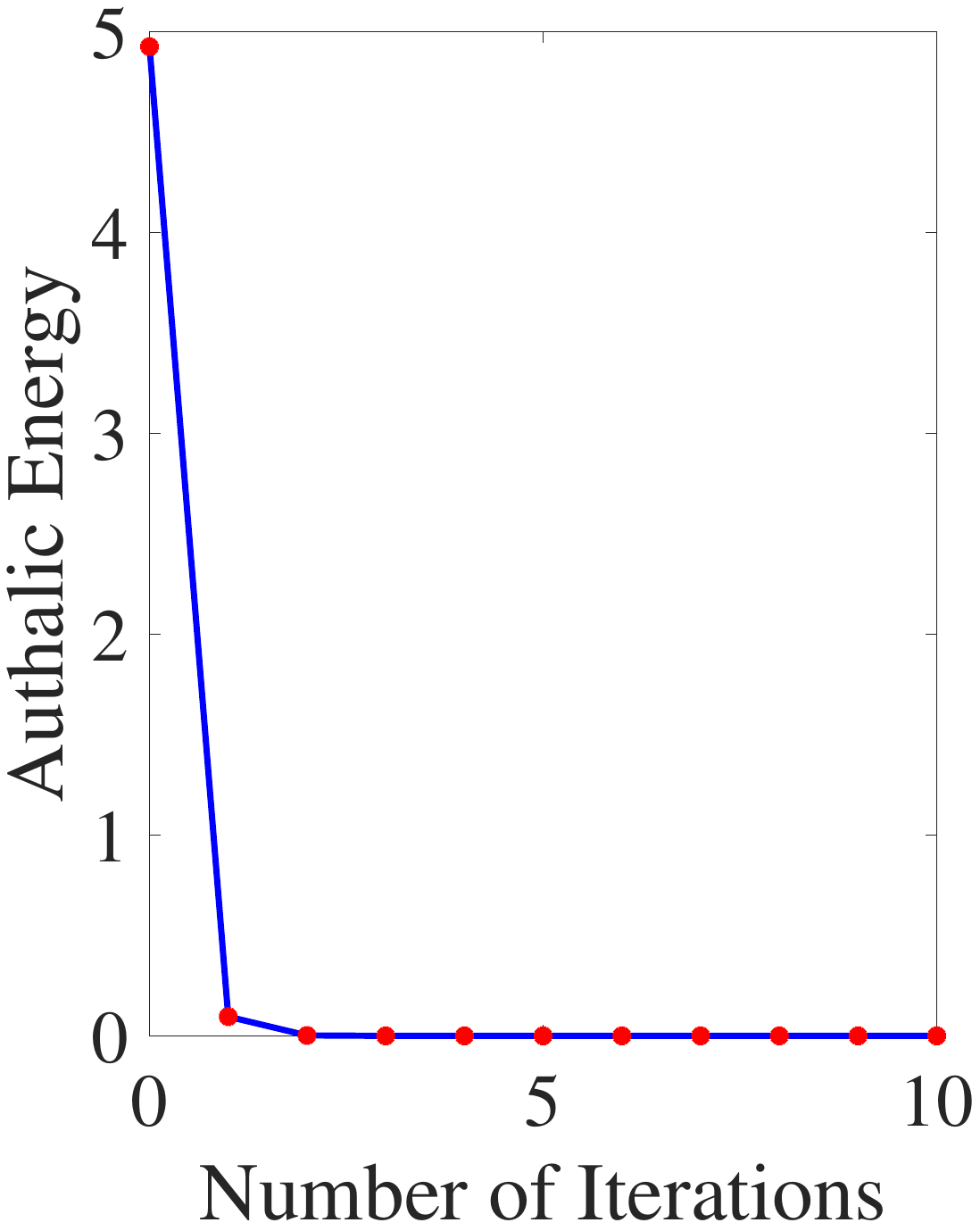} & 
\includegraphics[height=4.3cm]{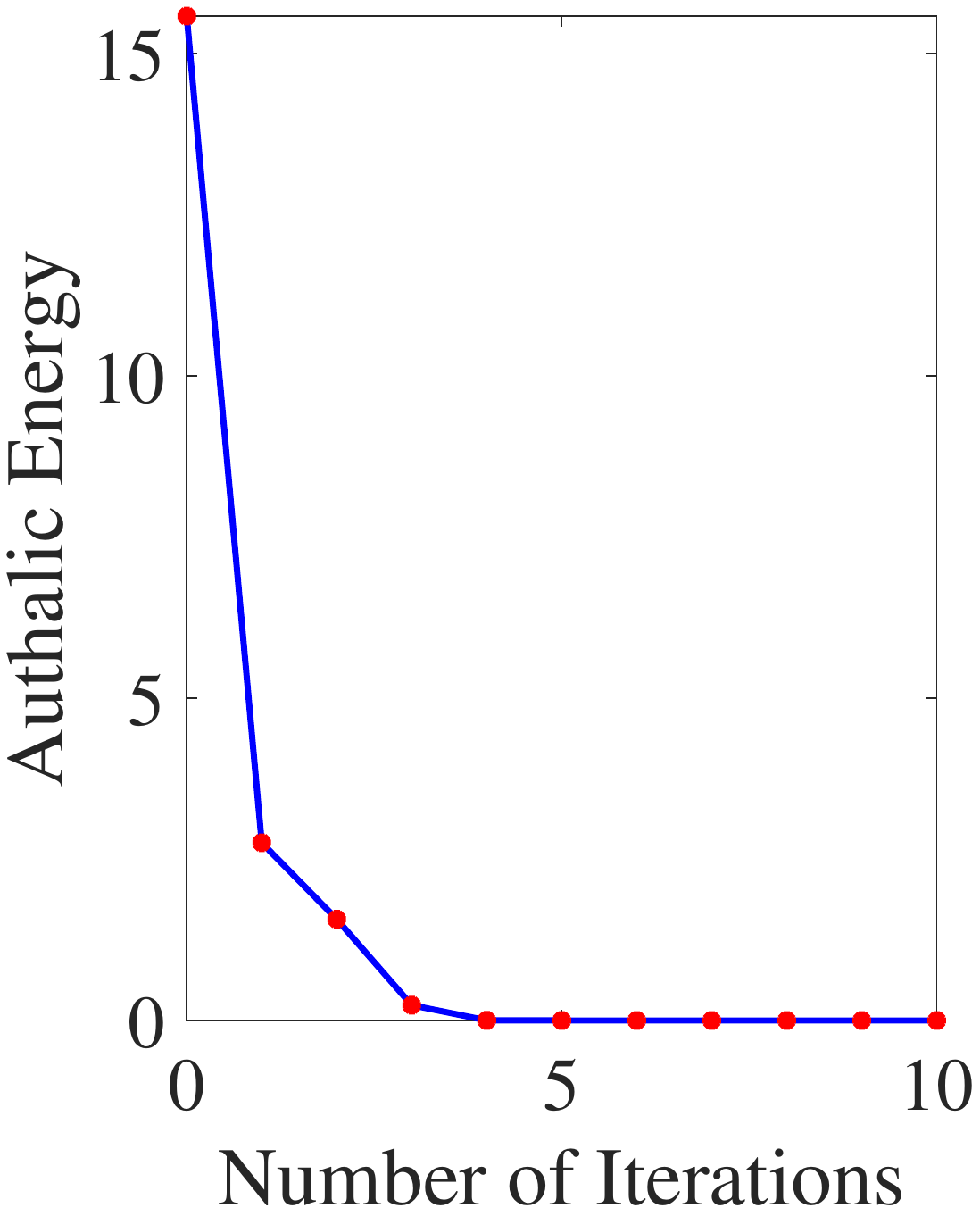}
\end{tabular}
\caption{The relationship between the number of iterations of Algorithm \ref{alg:SEM} and the authalic energy $E_A$ defined in \eqref{eq:E_A}.}
\label{fig:Ea}
\end{figure}

To show the area-preserving property of resulting mappings computed by the stretch energy minimization Algorithm \ref{alg:SEM}, in Figure \ref{fig:hist}, we demonstrate histograms of area ratios $R_A(f,\tau)$ defined in \eqref{eq:R_A}. We observe that the area ratios of most triangular faces are close to $1$, especially for the mesh models with larger numbers of vertices. These results indicate that the mappings computed by Algorithm \ref{alg:SEM} preserve the area well.

To quantify the area-preserving property of the resulting mappings, in Table \ref{tab:SEM}, we demonstrate the mean and the standard deviation (SD) of the area ratios
\begin{equation} \label{eq:R_A}
R_A(f,\tau) = \frac{|f(\tau)|}{|\tau|},
\end{equation}
of the resulting area-preserving mapping $f$ computed by Algorithm \ref{alg:SEM}, where the mean and the SD are taken over all triangular faces $\tau\in\F(\M)$. 
An area-preserving mapping $f$ has the property that 
$$
\underset{\tau\in\F(\M)}{\mathrm{mean}}R_A(f,\tau) = 1 
~\text{ and }~
\underset{\tau\in\F(\M)}{\mathrm{SD}}R_A(f,\tau) = 0. 
$$
From Table \ref{tab:SEM}, we observe that the mean and the SD of area ratios are reasonably close to $1$ and $0$, respectively, which indicates that the resulting mappings preserve the area well. 
In addition, the authalic energies of the resulting mappings are demonstrated in Table \ref{tab:SEM} as well. We see that the authalic energies of mappings are close to zero, from Corollary \ref{cor:3}, these mappings are reasonably close to area-preserving. 
Furthermore, the computational time cost of each mapping is also demonstrated in Table \ref{tab:SEM}. 
It would cost less than 80 seconds on a personal laptop to compute an area-preserving mapping of a mesh model of roughly 1 million vertices by Algorithm \ref{alg:SEM}, which is quite efficient.

\begin{figure}
\centering
\begin{tabular}{cccc}
Chinese Lion & Femur & Max Planck & Left Hand\\
\includegraphics[height=4.3cm]{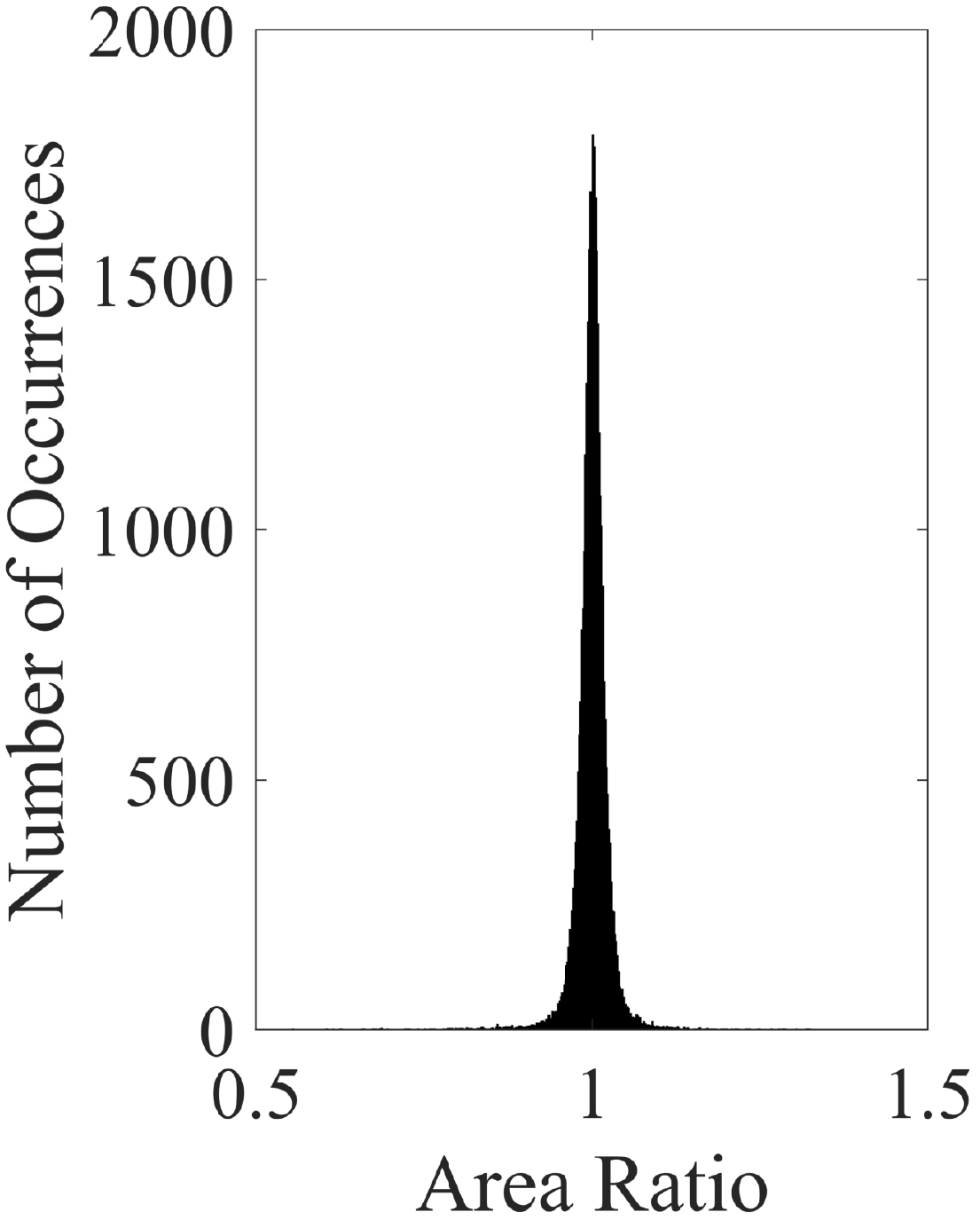} & 
\includegraphics[height=4.3cm]{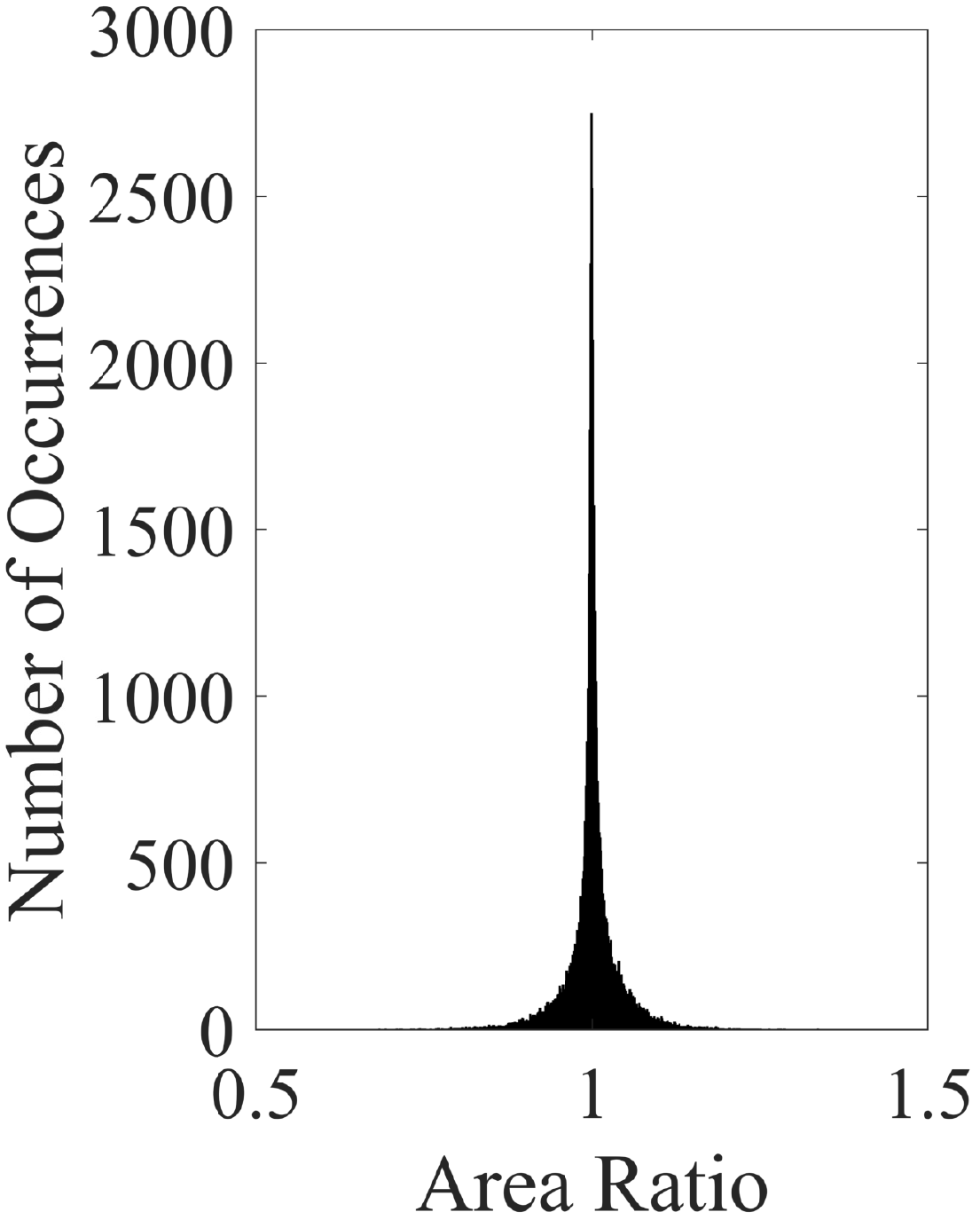} & 
\includegraphics[height=4.3cm]{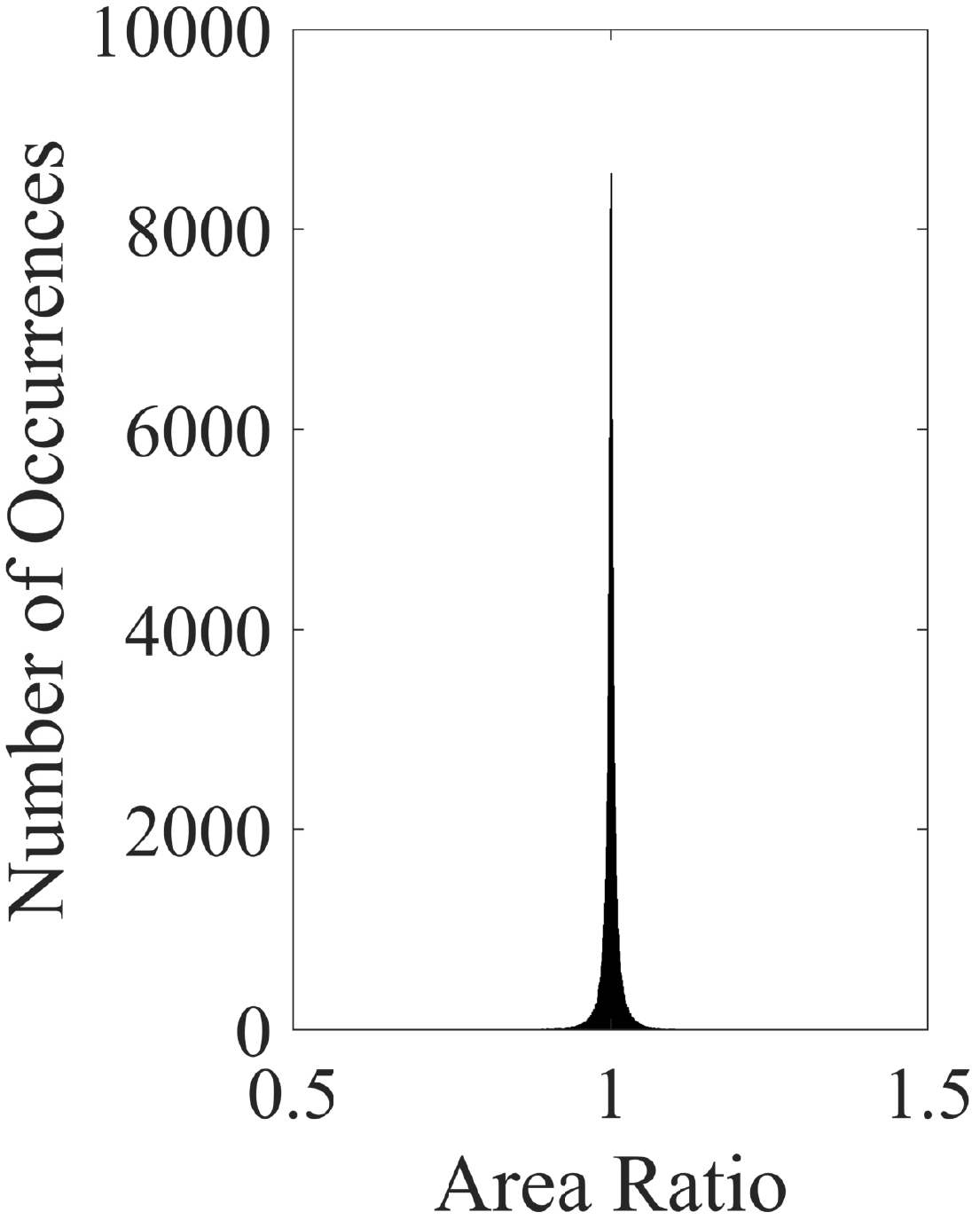} &
\includegraphics[height=4.3cm]{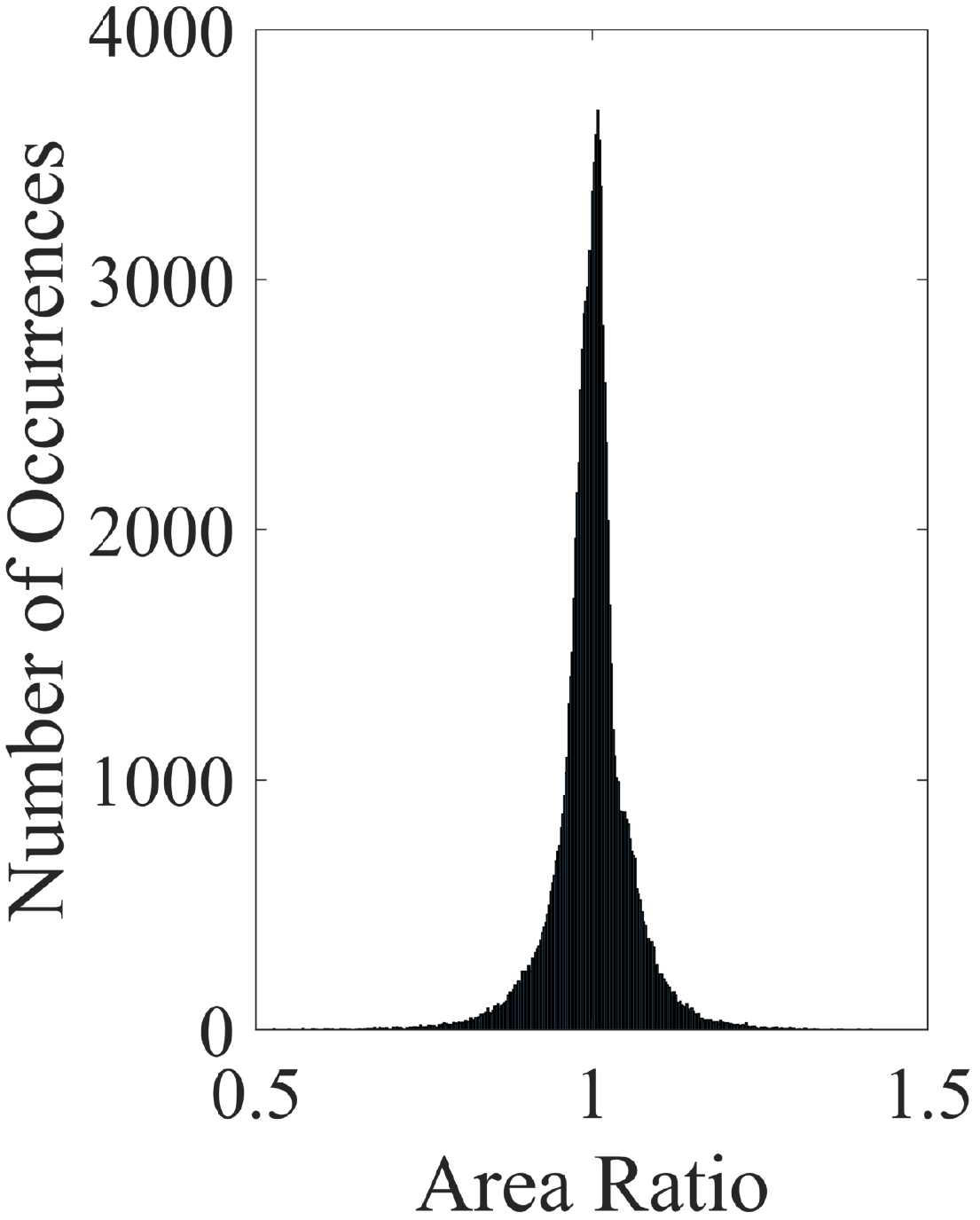} 
\\[0.5cm]
Knit Cap Man  & Bimba Statue & Buddha & Nefertiti Statue \\
\includegraphics[height=4.3cm]{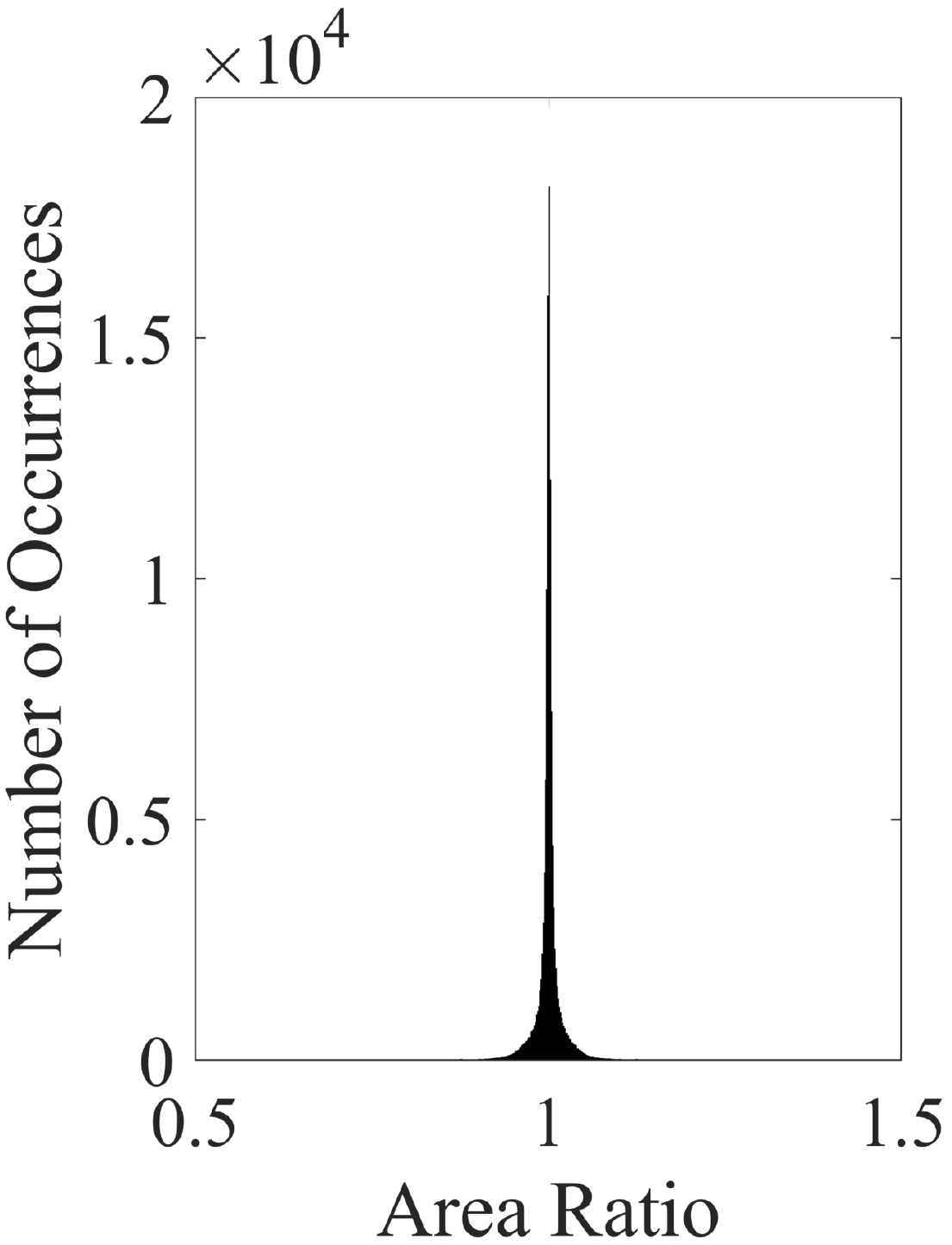} &
\includegraphics[height=4.3cm]{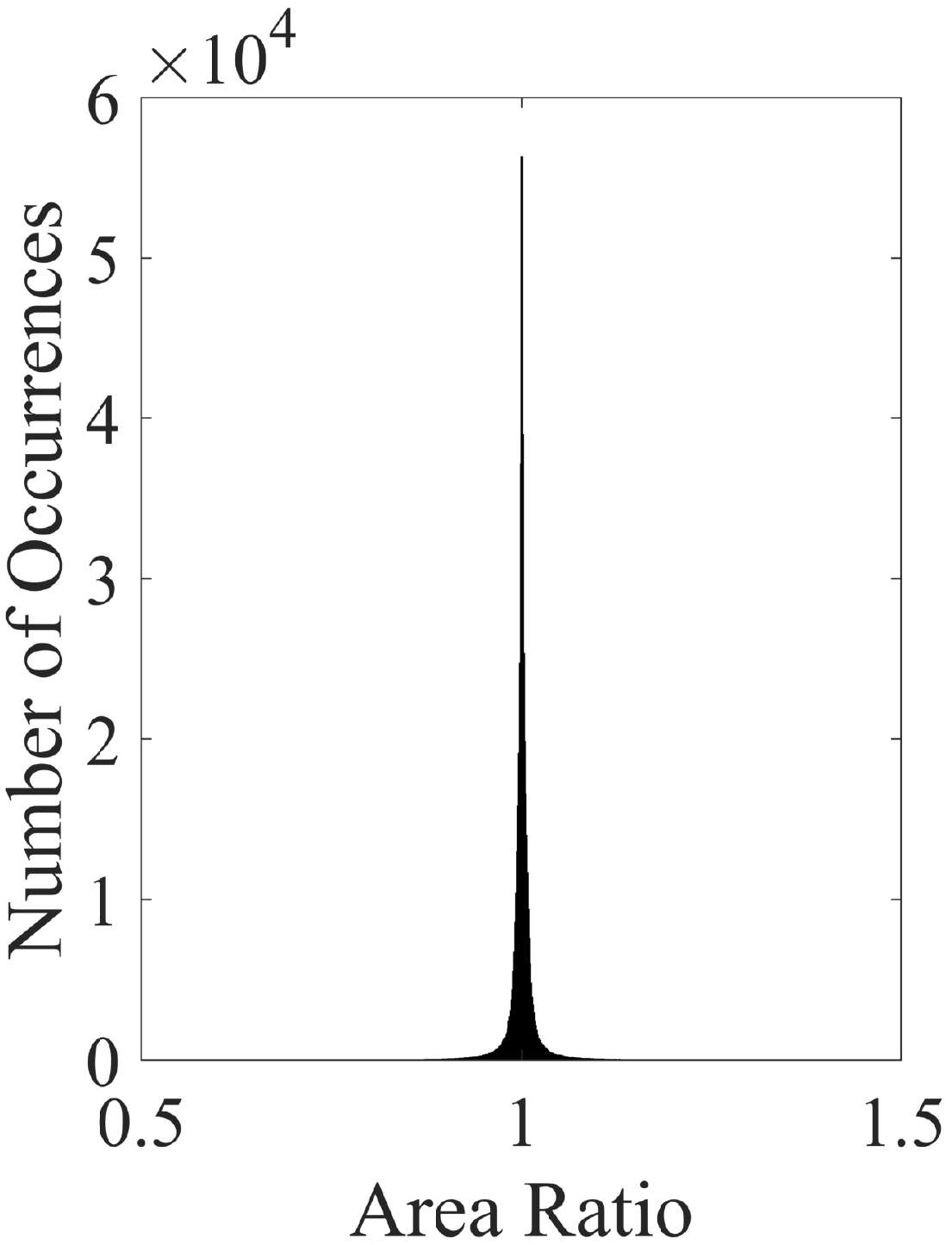} &
\includegraphics[height=4.3cm]{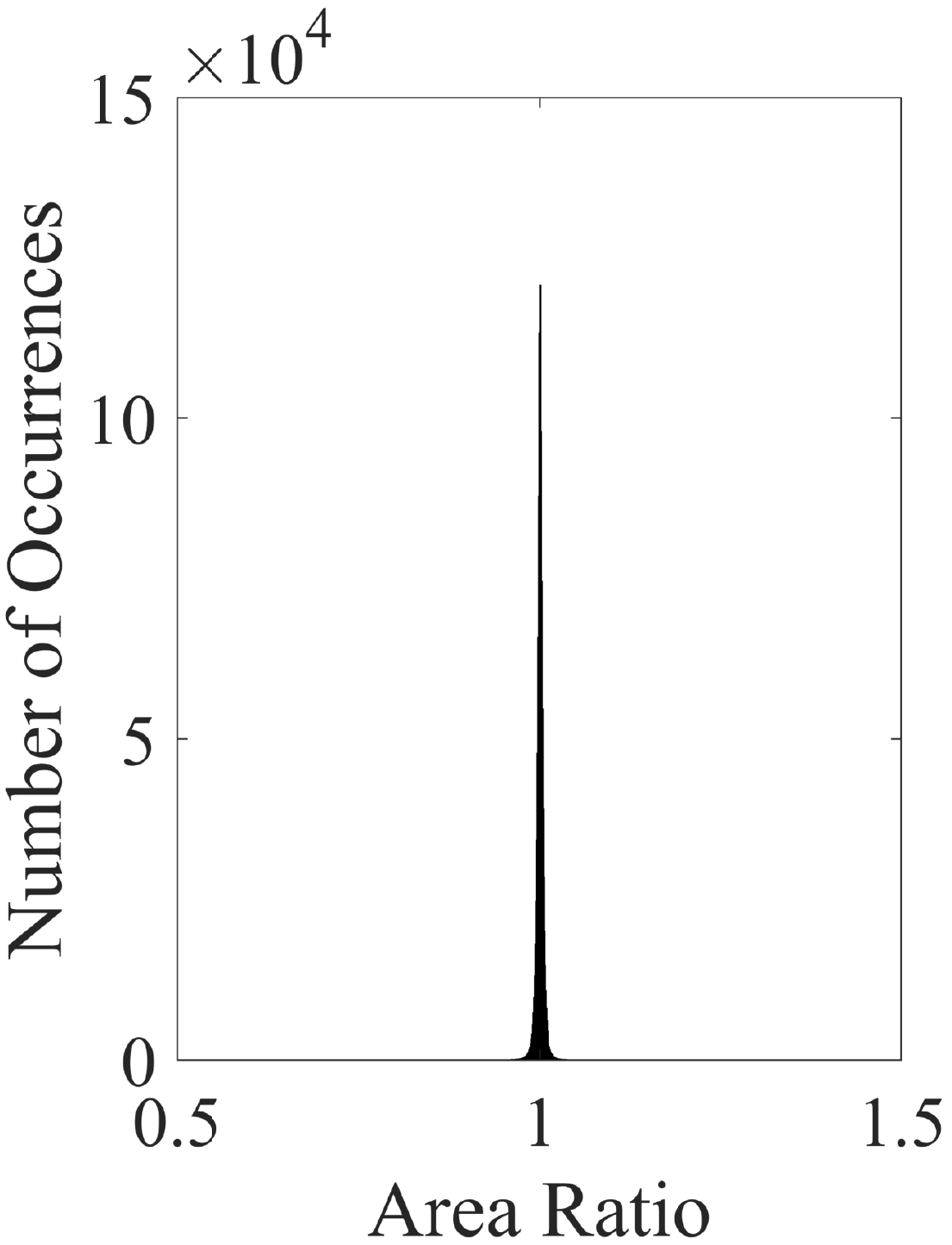} & 
\includegraphics[height=4.3cm]{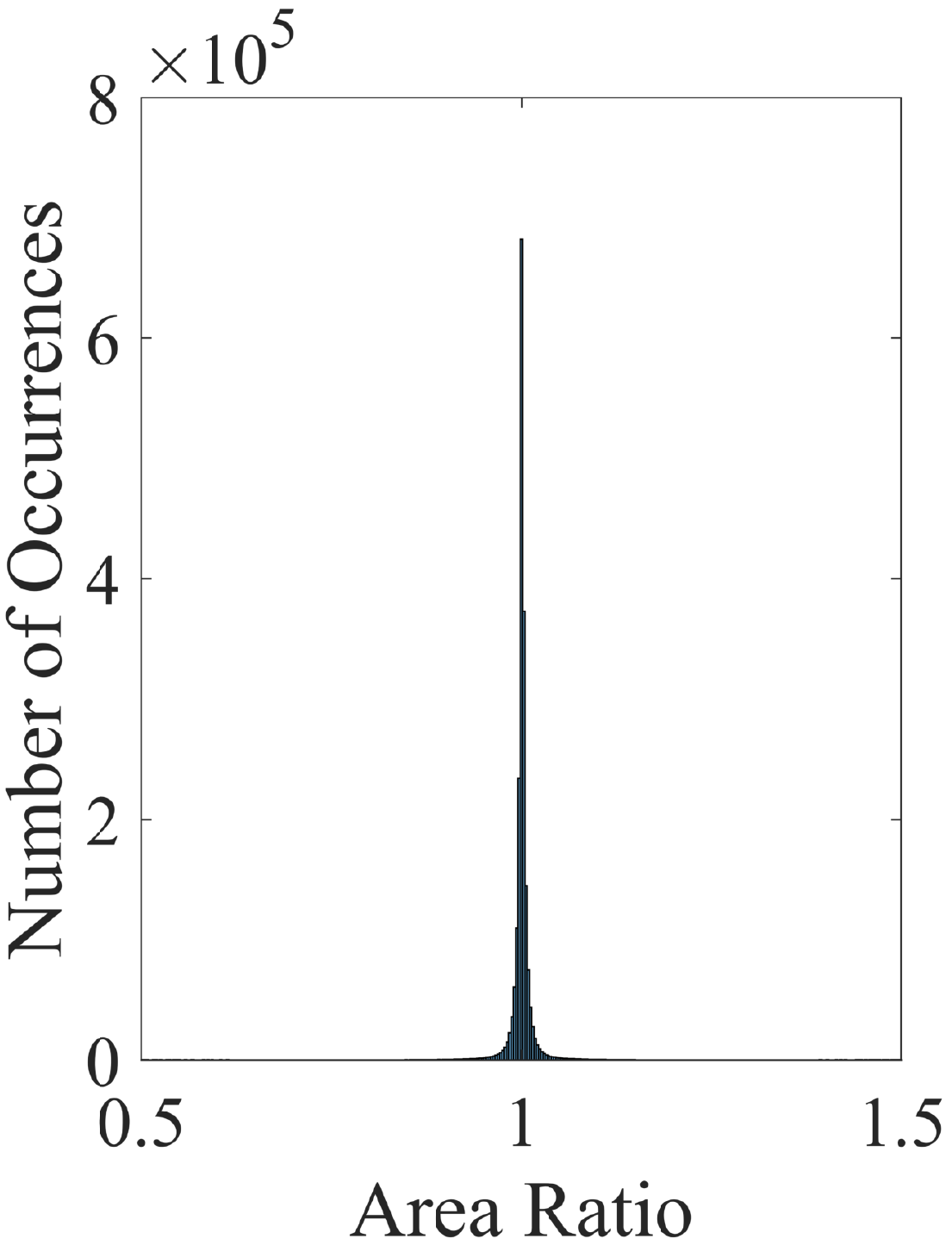}
\end{tabular}
\caption{Histograms of area ratios $R_A(f,\tau)$ of area-preserving mappings of benchmark mesh models computed by the stretch energy minimization Algorithm \ref{alg:SEM}.}
\label{fig:hist}
\end{figure}

\begin{table}[]
\centering
\begin{tabular}{lrrrrrr}
\hline 
\multirow{2}{*}{Model Name} & \multirow{2}{*}{\# Faces} & \multirow{2}{*}{\# Vertices} & \multicolumn{2}{c}{Area Ratio} & \multirow{2}{*}{Authalic Energy} & Time \\ 
&&& Mean & SD & & Cost
\\
\hline 
Chinese Lion     &   34,421 &  17,334 & 0.9999 & 0.0252 & $6.0610\times 10^{-4}$ & 0.78 \\ 
Femur            &   43,301 &  21,699 & 1.0000 & 0.0381 & $1.4043\times 10^{-3}$ & 1.11 \\ 
Max Planck       &   82,977 &  41,588 & 1.0000 & 0.0135 & $1.7194\times 10^{-4}$ & 2.18 \\ 
Left Hand        &  105,780 &  53,011 & 1.0000 & 0.0564 & $3.2248\times 10^{-3}$ & 2.61 \\ 
Knit Cap Man     &  118,849 &  59,561 & 1.0000 & 0.0161 & $2.5890\times 10^{-4}$ & 3.12 \\ 
Bimba Statue     &  433,040 & 216,873 & 1.0000 & 0.0167 & $2.7311\times 10^{-4}$ & 14.48 \\ 
Buddha           &  945,722 & 473,362 & 1.0000 & 0.0053 & $2.7495\times 10^{-5}$ & 41.30 \\ 
Nefertiti Statue &1,992,801 & 996,838 & 1.0000 & 0.0213 & $1.0074\times 10^{-3}$ & 75.48 \\ 
\hline 
\end{tabular}
\caption{The mean and SD of area ratios \eqref{eq:R_A}, the authalic energy \eqref{eq:E_A} and the computational time cost of the area-preserving mapping computed by the stretch energy minimization Algorithm \ref{alg:SEM}.}
\label{tab:SEM}
\end{table}

\section{Concluding remarks}
\label{sec:CR}

In this paper, we demonstrate the theoretical foundation of the stretch energy minimization for the computation of area-preserving mappings, including the derivation of a neat closed-form formulation of the gradient, the geometric interpretation of the stretch energy, and the proof of the minimum value occurs only at area-preserving mappings. 
Numerical experiments are demonstrated to show the effectiveness and accuracy of the stretch energy minimization for the computation of area-preserving mapping of simplicial surfaces. 
The derivation can be generalized to the volumetric stretch energy minimization for volume-preserving mappings of simplicial $3$-complexes, which is our future topic to investigate.

\paragraph{Acknowledgement}
The work of the author was partially supported by the Ministry of Science and Technology under grant numbers 109-2115-M-003-010-MY2 and 110-2115-M-003-014-, the National Center for Theoretical Sciences, the ST Yau Center in Taiwan, and the Ministry of Education.

\bibliographystyle{abbrv}
\bibliography{reference}

\end{document}